\documentclass[reqno,12pt]{amsart}
\usepackage{amsmath,amsthm,amssymb,mathtools,xcolor,dsfont,graphicx,tikz}
\usepackage[shortlabels]{enumitem}
\usepackage[utf8]{inputenc}
\usepackage[a4paper,margin=3.0truecm]{geometry}
\usepackage{subcaption}
\mathtoolsset{showonlyrefs}
\usepackage{hyperref}

\newcommand*{\R}{\mathbb{R}}

\newcommand{\cA}{\mathcal{A}}

\newcommand{\cM}{\mathcal{M}}
\newcommand{\cP}{\mathcal{P}}

\newcommand{\be}{\begin{equation}}
\newcommand{\ee}{\end{equation}}

\renewcommand{\d}{\,\mathrm{d}}
\newcommand*{\eqset}{\coloneqq}
\newcommand{\piv}{\ensuremath{\mathrm{piv}}}

\DeclareMathOperator*{\argmin}{argmin}

\DeclareMathOperator{\spt}{spt}

\DeclareMathOperator{\diam}{diam}

\DeclareMathOperator{\id}{id}
\DeclareMathOperator{\inter}{int}

\numberwithin{equation}{section}
\theoremstyle{plain}

\newtheorem{theorem}{Theorem}[section]
\newtheorem{corollary}[theorem]{Corollary}
\newtheorem{proposition}[theorem]{Proposition}
\newtheorem{lemma}[theorem]{Lemma}

\theoremstyle{remark}
\newtheorem{remark}[theorem]{Remark}
\newtheorem{example}[theorem]{Example}

\newcommand{\tfi}{\widetilde{\varphi}}
\newcommand{\eps}{\varepsilon}
\newcommand{\tB}{\widetilde{B}}
\newcommand{\tA}{\widetilde{A}}
\newcommand{\iK}{\inter(K)}
\newcommand{\mui}{\mu^{\inter}}
\newcommand{\mub}{\mu^{\mathrm{bd}}}

\title[Monge-Kantorovich interpolation and parking problem]{Monge-Kantorovich interpolation with constraints and application to a parking problem}
\author{Giuseppe Buttazzo 
\and
Guillaume Carlier
\and
Katharina Eichinger 
}
\date{August 2023}

\begin{document}
\maketitle

\begin{abstract}
We consider optimal transport problems where the cost for transporting a given probability measure $\mu_0$ to another one $\mu_1$ consists of two parts: the first one measures the transportation from $\mu_0$ to an intermediate (pivot) measure $\mu$ to be determined (and subject to various constraints), and the second one measures the transportation from $\mu$ to $\mu_1$. This leads to Monge-Kantorovich interpolation problems under constraints for which we establish various properties of the optimal pivot measures $\mu$. Considering the more general situation where only some part of the mass uses the intermediate stop leads to a mathematical model for the optimal location of a parking region around a city. Numerical simulations, based on entropic regularization, are presented both for the optimal parking regions and for Monge-Kantorovich constrained interpolation problems. 
\end{abstract}



\section{Introduction}\label{sec:intro}

We consider optimal transport problems where a given probability measure $\mu_0$ in $\R^d$ has to be transported to a given probability measure $\mu_1$ with minimal transportation cost. This cost consists of two parts: the first one measures the transportation from $\mu_0$ to an intermediate measure $\mu$, to be determined, and the second one measures the transportation from $\mu$ to $\mu_1$. This situation occurs in some applications, where the transport of $\mu_0$ to $\mu_1$ is not directly made but the possibility of an intermediate stop is taken into account. The two parts are described by the Monge-Kantorovich functionals $W_{c_0}(\mu_0,\mu)$ and $W_{c_1}(\mu,\mu_1)$ respectively, where for every pair of probabilities $\rho_0$, $\rho_1$ we set
\be\label{defwasser}
W_c(\rho_0,\rho_1)=\inf\Big\{\int_{\R^d\times\R^d} c(x,y)\,\d\gamma(x,y)\ :\ \gamma\in\Pi(\rho_0,\rho_1)\Big\}.
\ee
Here
$$\Pi(\rho_0,\rho_1)\eqset\left\{\gamma\in\cP(\R^d\times\R^d)\ :\ {\pi_i}_\#\gamma=\rho_i,\ i=0,1\right\}$$
is the set of transport plans between $\rho_0$ and $\rho_1$, denoting by $\pi_i:\R^d\times\R^d\to\R^d$ ($i=0,1$) the projections on the first and second factor respectively, ${\pi_i}_\# \gamma$ are the marginals of $\gamma$, so that a probability measure $\gamma$ on $\R^d\times \R^d$ belongs to $\Pi(\rho_0,\rho_1)$ when
$$\begin{cases}
{\pi_0}_\#\gamma(A)=\gamma(A\times\R^d)=\rho_0(A),\\
{\pi_1}_\#\gamma(A)=\gamma(\R^d\times A)=\rho_1(A)
\end{cases}\qquad\text{for all Borel set }A\subset\R^d.$$
Some extra constraints on the {\it pivot} measure $\mu$ can be added, as for instance:
\begin{itemize}
\item location constraints, where the support of $\mu$, $\spt\mu$, is required to be contained in a given region $K\subset\R^d$;
\item density constraints, where the measure $\mu$ is required to be absolutely continuous and with a density not exceeding a prescribed function $\phi$.
\end{itemize}

 Without additional constraint on the measure $\mu$, the minimization of $W_{c_0}(\mu_0,.)+ W_{c_1}(., \mu_1)$, or its generalizations to more than two prescribed measures, arise in different applied settings such as multi-population matching \cite{ce} or Wasserstein barycenters \cite{AC}. In particular, in the quadratic case where $c_0(x,y)=c_1(x,y)=\vert x-y\vert^2$, minimizers of $W_{c_0}(\mu_0,.)+ W_{c_1}(., \mu_1)$ are the midpoints of McCann's displacement interpolation \cite{McCann} between $\mu_0$ and $\mu_1$ i.e. geodesics for the quadratic Wasserstein metric\footnote{As kindly pointed out to us by a referee, naming optimal transport distances after Wasserstein is controversial and historically incorrect, eventhough the use of the name is widely spread in the literature, we preferred to mostly use Monge-Kantorovich instead in the present paper.}. Density constraints are important to model congestion effects as in the seminal crowd motion model of Maury, Roudneff-Chupin and Santambrogio \cite{MRS}. A first goal of the present paper is to investigate the effect of location and density constraints on such Monge-Kantorovich interpolation problems. Let us also mention that the minimization of $W_{c_0}(\mu_0, \mu)$ with respect to $\mu$ in a class of measures which are singular with respect to $\mu_0$ was addressed in \cite{BCL} whereas the parallel case where the density constraint appears in the definition of congestion penalization for singular measures was studied in \cite{LMSS} and \cite{Sarrazin}.

\smallskip

A second goal of the paper is to investigate a more general class of problems as a mathematical model for the optimal location of a parking region around a city. In this context, one is given two probability measures $\nu_0$ and $\nu_1$, which may be interpreted as a distribution of residents and a distribution of services respectively. A resident living at $x_0$ reaching a service located at $x_1$ may either walk directly to $x_1$ for the cost $c_1(x_0, x_1)$ or drive to an intermediate parking location $x$ and then walk from $x$ to $x_1$ paying the sum $c_0(x_0,x)+c_1(x, x_1)$. In this model, detailed in Section \ref{sec:parking}, the pivot/parking measure $\mu$ may have total mass less than $1$, and one may decompose $\nu_0$ and $\nu_1$ as $\nu_i=\nu_i-\mu_i+ \mu_i$ with $0 \leq \mu_i \leq \nu_i$ denoting the \emph{driving} part of $\nu_i$ and the unknowns $\mu_0$, $\mu$ and $\mu_1$ (with same total mass) should minimize the overall cost $W_{c_1}(\nu_0-\mu_0, \nu_1-\mu_1)+ W_{c_0}(\mu_0, \mu)+ W_{c_1}(\mu, \mu_1)$ subject to possible additional location and density constraints on $\mu$. Let us remark that if $(\mu_0, \mu_1, \mu)$ solves this parking problem, then $\mu$ solves the corresponding Monge-Kantorovich interpolation problem i.e. minimizes $W_{c_0}(\mu_0, .)+ W_{c_1}(., \mu_1)$ so that the qualitative properties established in Sections \ref{sec:distance} and \ref{sec:convex} will be directly applicable to optimal parking measures. We have chosen, as an application of our results, a model for determining the optimal location of parking areas around a city, but other models in different fields use similar frameworks and can be found in the literature: we quote for instance \cite{PR} and \cite{LPR}, where the transport between singular measures is used to model the behavior of biological membranes.
 
\smallskip

In Section \ref{sec:wasser}, we consider the general optimization problem \eqref{eq:transportwcar} and after solving an explicit example, we prove existence and discuss uniqueness of solutions. Dual formulations are introduced in Section \ref{sec:alter-form}. In Section \ref{sec:distance}, the particular case of distance-like costs is studied, while Section \ref{sec:convex} deals with the case of strictly convex cost functions, in these sections we study various qualitative properties of the solutions, in particular their integrability. 
In Section \ref{sec:parking}, we study a problem related to the optimization of a parking area. Finally, in Section \ref{sec:numeric}, we present some numerical simulations thanks to an entropic approximation scheme and compare the solutions of interpolation and parking problems.

\section{Monge-Kantorovich interpolation with constraints}\label{sec:wasser}

Let $\mu_0,\mu_1\in\cP(\R^d)$ be two probabilities with compact support, and let $c_0,c_1:\R^d\times\R^d\to\R_+$ be two continuous cost functions. For a class $\cA\subset\cP(\R^d)$ we are interested in solving the optimization problem
\be\label{eq:transportwcar}
\inf\big\{W_{c_0}(\mu_0,\mu)+W_{c_1}(\mu,\mu_1)\ :\ \mu\in\cA\big\}.
\ee
Here $W_{c_i}(\rho_0,\rho_1)$ denotes the value of the optimal transport problem between two measures $\rho_0,\rho_1\in\cP(\R^d)$, obtained by means of the Monge-Kantorovich functionals defined in \eqref{defwasser}. In order to simplify the presentation, by an abuse of notation, if $\rho$ is a measure and $\phi$ is a nonnegative Lebesgue integrable function, by $\rho\le\phi$ we mean that $\rho$ is is absolutely continuous and its density, again denoted $\rho$, satisfies $\rho \le\phi$ Lebesgue a.e. ; also, all the integrals with no domain of integration explicitly defined are intended on the whole $\R^d$.

Typical cases for the class $\cA$ of admissible choices are:
\begin{enumerate}[(i)]
\item no constraint, that is $\cA=\cP(\R^d)$;\label{case:Noconstraints}
\item location constraints, that is $\cA=\cP(K)$ for a nonempty compact subset $K$ of $\R^d$;\label{case:Location}
\item density constraints, that is $\cA=\{\rho\in\cP_{\mathrm{ac}}(\R^d):~\rho\le\phi\}$ for an $L^1$-function $\phi: \R^d\to\R_+$ with compact support and $\int\phi\,dx>1$.\label{case:Densityconstr}
\end{enumerate}

\subsection{Explicit one-dimensional examples}

Before going to the general case, let us illustrate our problem in a simple one-dimensional case, where optimal solutions can be easily recovered by explicit computations.

\begin{example}
Consider the one-dimensional case and the measures
$$\mu_0(x)=\mathds{1}_{[0,1]}(x),\qquad\mu_1(x)=\mathds{1}_{[5,6]}(x).$$
We first look at the case where the cost functions are given by distances:
$$c_0(x,y)=(1-t)|x-y|,\qquad c_1(x,y)=t|x-y|\qquad\text{ with }t\in]0,1[.$$
The following results can be easily seen by rephrasing the problem in terms of the distribution functions $f,f_0,f_1$ of the probabilities $\mu,\mu_0,\mu_1$ (see for instance Chapter 2 of \cite{san}):
$$\min\left\{\int_0^6(1-t) |f_0-f|+t|f-f_1|\,dx\ :\ f\text{ nondecreasing, }f(0)=0,\ f(6)=1\right\}$$
with the constraints
\begin{enumerate}[(i)]
\item no additional constraint;
\item $\spt f'\subset[2,4]$;
\item $f'\le\theta\mathds{1}_{[2,4]}$.
\end{enumerate}
Since $f_1 \le f_0$, it is easy to see that in the minimization above, one can always assume that $f_1\le f\le f_0$ and then remove the absolute values and minimize under the constraint that $f$ is nondecreasing and $f_1\le f\le f_0$. We then have:
\begin{enumerate}[(i)]
\item In the absence of constraints, this becomes the problem of finding the Wasserstein median between $\mu_0$ and $\mu_1$ (see \cite{CCE} for more on Wasserstein medians). In particular, the optimal solutions $\mu$ are characterized as follows:
\begin{itemize}
\item if $t>1/2$ (respectively $t<1/2$), the unique solution is given by $\mu=\mu_1$ (respectively $\mu=\mu_0$);
\item if $t=1/2$, any probability $\mu$ whose distribution function $f$ is between the two distribution functions $f_0$ and $f_1$ of $\mu_0$ and $\mu_1$, in the sense that
$$f_1(x)\le f(x)\le f_0(x)\qquad\text{for all }x\in\R,$$
is a minimizer.
\end{itemize}
\item In the case of the location constraint $K=[2,4]$, we observe a similar threshold effect:
\begin{itemize}
\item if $t>1/2$ (respectively $t<1/2$), the unique solution is given by $\mu=\delta_4$ (respectively $\mu=\delta_2$);
\item if $t=1/2$, then any probability measure supported on $K$ is a solution.
\end{itemize}
\item In the case of density constraint $\phi(x)\eqset\theta\mathds{1}_{[2,4]}(x)$ with $\theta>1/2$ we have:
\begin{itemize}
\item if $t>1/2$ (respectively $t<1/2$), the unique solution is given by $\mu=\theta\mathds{1}_{[4-1/\theta,4]}$ (respectively $\mu=\theta\mathds{1}_{[2,2+1/\theta]}$);
\item if $t=1/2$ any probability measure satisfying the constraint is a solution.
\end{itemize}
\end{enumerate}

The example above relies on the fact that for distance-like costs, optimality somehow forces the triangular inequality to be saturated in dimension 1. We will investigate this phenomenon further in Section \ref{sec:distance}.

We consider now strictly convex cost functions: as a prototype we take, with the same measures $\mu_0$ and $\mu_1$ above,
$$c_0(x,y)=(1-t)|x-y|^2,\quad c_1(x,y)=t|x-y|^2\qquad\text{with }t\in(0,1).$$
Also this case can be rephrased in terms of the so-called {\it pseudo-inverse} $g,g_0,g_1$ (see for instance Definition 2.1 of \cite{san}) of the distribution functions $f,f_0,f_1$ as:
$$\min\left\{\int_0^1(1-t)(g-g_0)^2+t(g_1-g)^2ds\ :\ g\text{ nondecreasing}\right\}$$
with the constraints
\begin{enumerate}[(i)]
\item no additional constraint;
\item $g([0,1])\subset[2,4]$;
\item $g'\ge1/\theta$ and $g([0,1])\subset[2,4]$.
\end{enumerate}
This implies:
\begin{enumerate}[(i)]
\item In the unconstrained case the solution simply corresponds to the Wasserstein-geodesic from $\mu_0$ to $\mu_1$ at time $t\in(0,1),$ or equivalently the weighted barycenter. It is given by
$$\mu_t(x)\eqset\mathds{1}_{[5t,1+5t]}(x).$$
\item Take the constraint $K=[2,4]$, as above. Here the solution depends on the location of the unconstrained geodesic $\mu_t$. We present a few cases (the other ones are clear by symmetry)
\begin{itemize}
\item if $t\le\frac{1}{5}$ the support of $\mu_t$ is contained in $[0,2]$, hence the optimal solution is $\delta_2$;
\item if $\frac{1}{5}<t<\frac{2}{5}$ the optimal solution is $\mu=(2-5t)\delta_2+\mathds{1}_{[2,1+5t]}$;
\item if $\frac{2}{5}\le t\le\frac{3}{5}$ the support of $\mu_t$ is contained in $[2,4]$, hence the solution is simply $\mu_t$.
\end{itemize}
\item Take the function $\phi(x) \eqset \theta \mathds{1}_{[2,4]}(x)$ with $1 > \theta > \frac{1}{2}$. The solution depends again on the location of the unconstrained geodesic $\mu_t.$ We have the following cases (remaining cases are again obtained by symmetry)
\begin{itemize}
\item if $t\le\frac{1}{5}$ the support of $\mu_t$ is contained in $[0,2]$, hence the optimal solution is $\theta\mathds{1}_{[2,2+1/\theta]}$;
\item if $\frac{1}{5}<t<\frac{2}{5}$ the optimal solution is still $\mu = \theta\mathds{1}_{[2,2+1/\theta]}$;
\item if $\frac{2}{5}\le t\le\frac{3}{5}$ the support of $\mu_t$ is contained in $[2,4]$, but by the density constraint $\mu_t$ is not even feasible this time. So the solution is of the form $\theta\mathds{1}_{[a,b]}$ with $2\le a<b\le4$ and $b-a=1/\theta$.
\end{itemize}
\end{enumerate}
\end{example}

\subsection{Reformulation, existence, uniqueness}

Let us now come back to the constrained Monge-Kantorovich interpolation problem \eqref{eq:transportwcar} assuming that the measures $\mu_0$ and $\mu_1$ are compactly supported and the costs $c_0$ and $c_1$ are continuous and nonnegative, then by the direct method one directly gets:

\begin{lemma}
Assume either (ii): $\cA=\cP(K)$ with $K$ nonempty and compact  or (iii) $\cA:=\{\rho\in\cP(\R^d)\ :\ \rho\le\phi\}$ with $\phi\in L^1$ compactly supported and $\int\phi\d x\ge1$. Then problem \eqref{eq:transportwcar} admits a solution.
\end{lemma}

\begin{proof}
In both cases, one is left to optimize over probabilities over a fixed compact set, the sum of two Monge-Kantorovich terms which are weakly* lsc.
\end{proof}

In the unconstrained case where $\cA:=\cP(\R^d)$, one of course needs some coercivity in the problem. We shall therefore assume that there exists a compact subset of $\R^d$, denoted (again) by $K$, such that for every $(x_0, x_1)\in\spt(\mu_0)\times\spt(\mu_1)$ one has
\be\label{eq:coerciv-c}
\argmin_{x\in\R^d}\{c_0(x_0,x)+c_1(x,x_1)\}\mbox{ is nonempty and included in $K$}.
\ee
We then define, for $(x_0,x_1)\in\spt(\mu_0)\times\spt(\mu_1)$
$$c(x_0,x_1):=\inf_{x\in\R^d}\{c_0(x_0,x)+c_1(x,x_1)\}=\min_{x\in K}\{c_0(x_0,x)+c_1(x,x_1)\}.$$
In the following proposition, we show that the optimization problem \eqref{eq:transportwcar}, with $\cA=\cP(\R^d)$, is equivalent to the standard transport problem with cost $c$:
\be\label{eq:InfCostF}
\inf_{\gamma\in\Pi(\mu_0,\mu_1)}\int_{\R^d\times\R^d} c(x_0,x_1)\d\gamma(x_0,x_1)
\ee
which clearly admits a solution, since $c\in C(\spt(\mu_0)\times\spt(\mu_1))$. We easily deduce the existence of a solution to \eqref{eq:transportwcar} when $\cA=\cP(\R^d)$ as well as the fact that all solutions are supported by $K$.

We will denote by $\Pi(\mu_0,\mu,\mu_1)$ the set of transport plans in the variables $(x_0,x,x_1)$ with marginals $\mu_0,\mu,\mu_1$, and we denote by $\pi_{0,\piv}$, $\pi_{\piv,1}$, $\pi_{0,1}$ the projections on the first and second, second and third, first and third factors respectively.

\begin{proposition}\label{prop:MultimargForm}
Assume \eqref{eq:coerciv-c}.
\begin{itemize}
\item  Let $\gamma\in\Pi(\mu_0,\mu_1)$ solve \eqref{eq:InfCostF} and let $T:\spt(\mu_0)\times\spt(\mu_1)\to\R^d$ be measurable and such that
\[T(x_0,x_1)\in\argmin_{x\in K}\{c_0(x_0,x)+c_1(x,x_1) \}\qquad\forall(x_0,x_1)\in\spt(\mu_0)\times\spt(\mu_1).\]
Then   $T_\#\gamma$ is a solution of \eqref{eq:transportwcar} with $\cA=\cP(\R^d)$ and the optimal values of \eqref{eq:transportwcar} and \eqref{eq:InfCostF} coincide;
\item conversely, for any optimal solution $\mu$ of \eqref{eq:transportwcar}, consider optimal transport plans $\gamma_0\in\Pi(\mu_0,\mu)$ with respect to the cost $c_0$ and $\gamma_1\in\Pi(\mu,\mu_1)$ with respect to the cost $c_1$. Then there exists a plan $\tilde\gamma\in\Pi(\mu_0,\mu,\mu_1)$ with ${\pi_{0,\piv}}_\#\tilde\gamma=\gamma_0$ and ${\pi_{\piv,1}}_\#\tilde\gamma=\gamma_1$ such that ${\pi_{0,1}}_\#\tilde\gamma$ is optimal for \eqref{eq:InfCostF} and $c_0(x_0,x)+c_1(x,x_1)=c(x_0,x_1)$ on $\spt(\tilde\gamma)$ so that $\mu$ is supported by $K$.
\end{itemize}
The previous equivalence also holds between \eqref{eq:transportwcar} with $\cA=\cP(K)$ (with $K$ a given compact subset of $\R^d$) and \eqref{eq:InfCostF} with $c$ given by $c(x_0, x_1)=\min_{x\in K}\{c_0(x_0,x)+c_1(x,x_1)\}$. 
\end{proposition}

\begin{proof}
Let $\mu\in\cP(\R^d)$, $\gamma_0\in\Pi(\mu_0,\mu)$ and $\gamma_1\in\Pi(\mu,\mu_1)$; by the gluing Lemma (see Lemma 7.6 in \cite{villatopic}), there is a plan $\tilde\gamma\in\Pi(\mu_0,\mu,\mu_1)$ with ${\pi_{0,\piv}}_\#\tilde\gamma=\gamma_0$ and ${\pi_{\piv,1}}_\#\tilde\gamma=\gamma_1$. Hence, since $\gamma$ solves \eqref{eq:InfCostF} and ${\pi_{0,1}}_\#\tilde\gamma\in\Pi(\mu_0,\mu_1)$, we have
\begin{align*}
\int_{\R^d\times\R^d} c_0\d\gamma_0+\int_{\R^d\times\R^d} c_1\d\gamma_1
&=\int_{\R^d\times\R^d\times\R^d}\left\{c_0(x_0,x)+c_1(x,x_1)\right\}\d\tilde\gamma(x_0,x,x_1)\\
&\ge\int_{\R^d\times\R^d\times} c\d{\pi_{0,1}}_\#\tilde\gamma\ge\int_{\R^d\times\R^d} c\d\gamma\\
&=\int_{\R^d\times\R^d}\left\{c_0(x_0,T(x_0,x_1))+c_1(T(x_0,x_1),x_1)\right\}\d\gamma(x_0,x_1)\\
&\ge W_{c_0}(\mu_0,T_\#\gamma)+W_{c_1}(T_\#\gamma,\mu_1)
\end{align*}
which, taking the infimum with respect to $\gamma_0\in\Pi(\mu_0,\mu)$ and $\gamma_1\in\Pi(\mu,\mu_1)$, enable us to deduce that $T_\#\gamma$ solves \eqref{eq:transportwcar} as well as the equality of the optimal values of \eqref{eq:transportwcar} and \eqref{eq:InfCostF}.

Assume now that $\mu$ solves \eqref{eq:transportwcar} and consider optimal transport plans $\gamma_0\in\Pi(\mu_0,\mu)$ with respect to the cost $c_0$ and $\gamma_1\in\Pi(\mu,\mu_1)$ with respect to the cost $c_1$. Using again the gluing lemma we find $\tilde\gamma\in\Pi(\mu_0,\mu,\mu_1)$ with ${\pi_{0,\piv}}_\#\tilde\gamma=\gamma_0$ and ${\pi_{\piv,1}}_\#\tilde\gamma=\gamma_1$, and we then have
\begin{align*}
\inf \eqref{eq:InfCostF}
&=\inf \eqref{eq:transportwcar}=\int_{\R^d\times\R^d\times\R^d}\left\{c_0(x_0,x)+c_1(x,x_1)\right\}\d\tilde\gamma(x_0,x,x_1)\\
&\ge\int_{\R^d\times\R^d\times\R^d} c(x_0,x_1)\d\tilde\gamma(x_0,x,x_1)=\int_{\R^d\times\R^d\times} c\d{\pi_{0,1}}_\#\tilde\gamma.
\end{align*}
Therefore ${\pi_{0,1}}_\#\tilde\gamma$ is optimal for \eqref{eq:InfCostF} and $c_0(x_0,x)+c_1(x,x_1)=c(x_0,x_1)$ on $\spt(\tilde\gamma)$.
\end{proof}

In other words, the coercivity condition \eqref{eq:coerciv-c} ensures that we can replace $\cA=\cP(\R^d)$ by $\cA=\cP(K)$ in \eqref{eq:transportwcar} and therefore always optimize over probabilities over a fixed compact subset of $\R^d$. 

\begin{remark}\label{rem:uniqu}
We now discuss uniqueness. Letting $K$ be a nonempty compact subset of $\R^d$ and $\cA$ be a convex subset of $\cP(K)$, note first that $\mu \in \cA\mapsto W_{c_0}(\mu_0,\mu) $ and $\mu \in \cA\mapsto W_{c_1}(\mu,\mu_1) $ are convex (regardless of specific assumptions on the costs and the measures $\mu_0$ and $\mu_1$). If we further assume that
 $\mu_0$ is absolutely continuous and $c_0$ is locally Lipschitz and satisfies the twist condition, i.e. it is differentiable in the first coordinate and for every $x_0\in\spt(\mu_0)$
$$y\mapsto\nabla_{x_0} c_0(x_0,y)\qquad\hbox{is injective,}$$
 then we claim that 
\begin{equation}\label{eq:StrictConvW}
\mu\mapsto W_{c_0}(\mu_0,\mu)\qquad\hbox{is strictly convex.}
\end{equation}
This implies in particular the strict convexity of functional to be minimized in \eqref{eq:transportwcar}, and thus  the uniqueness of a minimizer.
The proof of strict convexity of \eqref{eq:StrictConvW} follows the same  lines as Proposition 7.19 in \cite{san}, we recall the argument for the sake of completeness. Indeed,  thanks to the twist condition and the regularity assumptions on $c_0$ and $\mu_0$, the optimal transport problem between $\mu_0$ and any $\mu\in \cA$ has a unique transport plan induced by a map, see Proposition 1.15 of \cite{san} and the discussion after. Assume that $(\mu, \widetilde{\mu}, t) \in \cA\times \cA\times (0,1)$ are such  that 
\[W_{c_0}(\mu_0, (1-t)\mu + t \widetilde{\mu}) = (1-t) W_{c_0}(\mu_0, \mu) + t W_{c_0}(\mu_0, \widetilde{\mu})\]
denoting by $T$ and $\widetilde{T}$ the optimal transport maps between $\mu_0$ and $\mu$ and $\mu_0$ and $\widetilde{\mu}$ respectively, we have  
\[W_{c_0}(\mu_0, (1-t)\mu + t \widetilde{\mu})=\int_{\R^d\times \R^d} c_0 \d  \gamma_t \mbox{ with } \gamma_t:=(\id, (1-t)T+ t \widetilde{T})_\#\mu_0\]  
and since $\gamma_t \in \Pi(\mu_0, (1-t)\mu + t \widetilde{\mu})$, we deduce that $\gamma_t$ is an  optimal plan between $\mu_0$ and  $(1-t)\mu + t \widetilde{\mu}$, which, by the twist condition and the absolute continuity of $\mu_0$  implies that $\gamma_t$ is induced by a map so that $T=\widetilde{T}$ $\mu_0$-a.e. hence $\mu=T_{\#}\mu_0= \widetilde{T}_{\#} \mu_0=\widetilde{\mu}$. This shows the announced strict convexity claim.  In particular, this argument gives uniqueness for smooth and strictly convex costs. Note that this also gives uniqueness for \ref{case:Location} and \ref{case:Densityconstr} in the case of concave costs, i.e. when $c_0(x,y)=l(|x-y|)$ for $l:\R_+\to\R_+$ strictly concave, increasing and differentiable on $(0,+\infty)$, if we assume $\mu_0$ absolutely continuous and for \ref{case:Location} $K\cap\spt\mu_0=\emptyset$, or for \ref{case:Densityconstr} $\spt(\phi)\cap\spt\mu_0=\emptyset$ (see \cite{GMC} or \cite{piazzoli_full_2015} for refinements and weaker conditions). All these arguments for uniqueness of course remain true if we replace the assumptions on $\mu_0$ and $c_0$ by similar assumptions on $\mu_1$ and $c_1$.
\end{remark}

\section{Dual formulations}\label{sec:alter-form}
\subsection{Location constraints}\label{subsec:loc-cont}

Thanks to the coercivity condition \eqref{eq:coerciv-c} any solution $\mu$ to \eqref{eq:transportwcar} with $\cA=\cP(\R^d)$ is necessarily concentrated on the compact set $K$, hence both cases (i) and (ii) can be formulated over $\cP(K)$. In this case, it can be convenient, to characterize solutions of the convex minimization problem \eqref{eq:transportwcar} by duality as follows. Given $\varphi\in C(K)$, define the $c_0$-transform of $\varphi$, $\varphi^{c_0}\in C(\spt(\mu_0))$ by
\begin{equation}\label{ctransf00}
\varphi_0^{c_0}(x_0):=\min_{x\in K} \{c_0(x_0, x)-\varphi(x)\}\qquad\forall x_0 \in \spt(\mu_0),
\end{equation}
and similarly define the $c_1$-transform of $\varphi$, $\varphi^{c_1}\in C(\spt(\mu_1))$ by
\begin{equation}\label{ctransf11}
\varphi_1^{c_1}(x_1):=\min_{x\in K}\{c_1(x,x_1)-\varphi(x)\}\qquad\forall x_1\in\spt(\mu_1).
\end{equation}
It follows from Theorem 3 in \cite{ce} (where the more general multi-marginal case is considered) that the minimum in \eqref{eq:transportwcar} coincides with the value of the dual:
\be\label{eq:DualCtrans}
\sup\left\{\int\varphi_0^{c_0}\d\mu_0+\int\varphi_1^{c_1}\d\mu_1\ :\ \varphi_0,\varphi_1\in C(K),\ \varphi_0+\varphi_1=0\right\},
\ee
and the supremum in \eqref{eq:DualCtrans} is attained. Moreover, if $\varphi_0$ and $\varphi_1$ solve \eqref{eq:DualCtrans}, then $\mu\in\cP(K)$ solves \eqref{eq:transportwcar} if and only if $\varphi_0$ is a Kantorovich potential between $\mu_0$ and $\mu$ and $\varphi_1$ is a Kantorovich potential (see \cite{villatopic} and \cite{san} for more on Kantorovich duality) between $\mu$ and $\mu_1$ i.e. there exist $(\gamma_0, \gamma_1)\in\Pi(\mu_0,\mu)\times\Pi(\mu,\mu_1)$ such that
\[\begin{split}
&\varphi_0(x)+\varphi_0^{c_0}(x_0)=c_0(x_0,x)\qquad\forall(x_0,x)\in\spt(\gamma_0),\\
&\varphi_1(x)+\varphi_1^{c_1}(x_1)=c_1(x,x_1)\qquad\forall(x,x_1)\in\spt(\gamma_1).
\end{split}\]
Defining the $c_0$-concave envelope of $\varphi_0$ and the $c_1$-concave envelope of $\varphi_1$ by
\[\begin{split}
&\tfi_0(x):=\min_{x_0\in\spt(\mu_0)}\{c_0(x_0,x)-\varphi_0^{c_0}(x_0)\},\\
&\tfi_1(x):=\min_{x_1\in\spt(\mu_1)}\{c_1(x,x_1)-\varphi_1^{c_1}(x_1)\},
\end{split}\]
one has $\tfi_0\ge\varphi_0$ and $\tfi_1\ge\varphi_1$ with an equality on $\spt(\mu)$ so that $\tfi_0+ \tfi_1\ge0$ with an equality on $\spt(\mu)$.

\subsection{Density constraint} 
We now consider case (iii) where there is a constraint on the density $\mu \leq \phi$, one can characterize minimizers by duality as follows:
 
\begin{proposition}\label{densityconstraintsol}
Consider \eqref{eq:transportwcar} in the case (iii) where there is a constraint on the density $\mu\le\phi$ with $\phi\in L^1(\R^d)$, $\phi\ge0$, $\int\phi\,dx>1$ and $\spt(\phi)$ compact (as well as $\spt(\mu_0)$ and $\spt(\mu_1)$). Then the value of \eqref{eq:transportwcar} coincides with the value of its (pre-)dual formulation
\be\label{dualdensityconstraint}
\sup_{\varphi_0, \varphi_1 \in C(\spt(\phi))^2}\int\varphi_0^{c_0} \d\mu_0+\int\varphi_1^{c_1} \d\mu_1+\int\min (\varphi_0+\varphi_1,0)\phi\,\d x
\ee
(where $\varphi_i^{c_i}$ are as in formulae \eqref{ctransf00}-\eqref{ctransf11} with $K$ replaced by $\spt(\phi)$).
Moreover, the supremum in \eqref{dualdensityconstraint} is attained. If $(\varphi_0, \varphi_1)$ solves \eqref{dualdensityconstraint}, then $\mu$ solves \eqref{eq:transportwcar} under the constraint $\mu\leq \phi$ if and only if there exist $\gamma_0\in \Pi(\mu,\mu_0)$ and $\gamma_1 \in \Pi(\mu, \mu_1)$ such that
\be\label{cn1dc}
\varphi_0(x)+ \varphi_0^{c_0}(x_0)=c_0(x_0, x),\qquad\forall (x_0,x) \in \spt(\gamma_0), 
\ee
\be\label{cn2dc}
\varphi_1(x)+ \varphi_1^{c_1}(x_1)=c_1(x, x_1),\qquad\forall (x,x_1) \in \spt(\gamma_1)
\ee
(so that $\gamma_0$ and $\gamma_1$ are optimal plans and $\varphi_0$ and $\varphi_1$ are Kantorovich potentials) and
\be\label{cn3dc}
\varphi_0+\varphi_1\ge0\mbox{ on }\spt(\phi-\mu),\qquad \varphi_0+\varphi_1\le0\mbox{ on }\spt(\mu).
\ee
\end{proposition}

\begin{proof}
The fact that the concave maximization problem \eqref{dualdensityconstraint} is the dual of \eqref{eq:transportwcar} under the constraint $\mu\leq \phi$ follows from the Fenchel-Rockafellar duality theorem and the Kantorovich duality formula. Indeed, we first have
\[\sup \eqref{dualdensityconstraint}=-\inf_{\varphi_0, \varphi_1 \in C(\spt(\phi))^2} F(\varphi_0, \varphi_1)+ G(-\varphi_0-\varphi_1)\]
where 
\[F(\varphi_0, \varphi_1):=-\int\varphi_0^{c_0} \d\mu_0-\int\varphi_1^{c_1} \d\mu_1, \; G(\varphi):=\int \max(\varphi, 0)\phi \; \d x, \; \varphi \in C(\spt(\phi)).\]
Note that $F$ and $G$ are convex and continuous for the uniform convergence topology and it is easy to see that $\sup \eqref{dualdensityconstraint}$ is finite (see the proof of existence of a solution to \eqref{dualdensityconstraint} below) so that by the Fenchel-Rockafellar duality Theorem, we have
\[\sup \eqref{dualdensityconstraint}=-\sup_{\mu \in C(\spt(\phi))^*}\{-F^*(\mu, \mu)-G^*(\mu)\}=\inf_{\mu \in C(\spt(\phi))^*}\{F^*(\mu, \mu)+G^*(\mu)\} \]
By Kantorovich duality formula (Proposition 1.11 in \cite{san}) we have (also see \cite{ce} for details)
\[F^*(\mu, \mu)=\begin{cases} W_{c_0}(\mu_0, \mu) + W_{c_1}(\mu, \mu_1) \mbox{ if $\mu \in \cP(\spt(\phi))$} \\ + \infty \mbox{ otherwise.} \end{cases}\]
and \[G^*(\mu)=\sup_{\varphi \in C(\spt(\phi))} \int \varphi \d \mu-\int \max(\varphi, 0)\phi \; \d x\]
and when $\mu \in \cP(\spt(\phi))$ (which is the case when $F^*(\mu, \mu)<+\infty$), the above maximization can be restricted to nonnegative functions $\varphi$ yielding
\[G^*(\mu)=\begin{cases} 0 \mbox{ if $\mu \leq \phi$}\\ + \infty \mbox{ otherwise.}\end{cases}\] 
We thus have
\[\sup \eqref{dualdensityconstraint}=\inf_{\mu \in  \cP(\spt(\phi)), \; \mu \leq \phi} W_{c_0}(\mu_0, \mu) + W_{c_1}(\mu, \mu_1) \]
and (up to extending $\mu$ by $0$ outside $\spt(\phi)$) the right-hand side of the previous equality is equivalent to \eqref{eq:transportwcar} in the case (iii) where there is a constraint on the density $\mu\le\phi$.
Let us now prove that \eqref{dualdensityconstraint} admits a solution. To see this we remark that the objective is unchanged when one replaces $(\varphi_0,\varphi_1)$ by $(\varphi_0+\lambda, \varphi_1-\lambda)$ where $\lambda$ is a constant. Moreover, replacing $\varphi_0$ and $\varphi_1$ by their $c_0/c_1$-concave envelopes defined for every $x\in \spt(\phi)$ by:
\be\label{cctrick}
\begin{split}
\tfi_0(x)&:=\min_{x_0\in\spt(\mu_0)}\{c_0(x_0,x)-\varphi_0^{c_0}(x_0)\}, \;\\ \tfi_1(x)
&:=\min_{x_1 \in \spt(\mu_1)} \{c_1(x, x_1)-\varphi_1^{c_1}(x_1)\}
\end{split}\ee
it is well-known that $\tfi_i\ge\varphi_i$ and $\tfi_i^{c_i}=\varphi_i^{c_i}$ for $i=0,1$ so that replacing $\varphi_i$ by $\tfi_i$ is an improvement in the objective of \eqref{dualdensityconstraint}, moreover the functions $\tfi_i$ have a uniform modulus of continuity inherited from the uniform continuity of $c_i$. From these observations, we can find a uniformly equicontinuous maximizing sequence $(\varphi_0^n, \varphi_1^n)_n$ for which $\min_{\spt(\phi)} \varphi_0^n=0$ so that $\varphi_0^n$ is also uniformly bounded. Since $\min(\varphi_1^n+ \varphi_0^n,0)\leq 0$, the fact that $(\varphi_0^n, \varphi_1^n)_n$ is a maximizing sequence together with the uniform bounds on $\varphi_0^n$ we get a uniform lower bound on $\int (\varphi_1^n)^{c_1} d\mu_1$ from which we easily derive a uniform upper bound on $\varphi_1^n$ thanks to \eqref{cctrick}. To show that $\varphi_1^n$ is also uniformly bounded from below, we observe that the quantity
$$\int(\varphi_1^n)^{c_1}d\mu_1+\int\min(\varphi_1^n+\varphi_0^n,0)\phi\,dx$$
is bounded from below and bounded from above by $C+(\int\phi\,dx-1) \min_{\spt(\phi)} \varphi_1^n$ for some constant $C$. Since $\int\phi\,dx>1$ this gives the desired lower bound. Having thus found a uniformly bounded and equicontinuous maximizing sequence, we deduce the existence of a solution to \eqref{dualdensityconstraint} from Arzel\`a-Ascoli theorem.

Let us now look at the optimality conditions which follow from the above duality. If $(\varphi_0, \varphi_1)$ solves \eqref{dualdensityconstraint}, then $\mu$ solves \eqref{eq:transportwcar} under the constraint $\mu\leq \phi$ if and only if
$$W_{c_0}(\mu_0, \mu)+ W_{c_1}(\mu, \mu_1) =\int \varphi_0^{c_0} \d \mu_0+ \int \varphi_1^{c_1} \d \mu_1+ \int \min (\varphi_0+\varphi_1,0)\phi.$$
If $\gamma_0$ (respectively $\gamma_1$) is an optimal plan for $c_0$ (resp. $c_1$) between $\mu_0$ and $\mu$ (resp. $\mu$ and $\mu_1$), we thus have
\[\begin{split}
&\int\varphi_0^{c_0}\d\mu_0+\int\varphi_1^{c_1}\d\mu_1+\int\min(\varphi_0+\varphi_1,0)\phi=\int c_0 \d \gamma_0+ \int c_1 \d \gamma_1\\
&\qquad\ge\int (\varphi_0^{c_0}(x_0)+\varphi_0(x)) \d \gamma_0(x_0, x)+\int(\varphi_1^{c_1}(x_1)+\varphi_1(x)) \d \gamma_1(x, x_1)\\
&\qquad=\int\varphi_0^{c_0} \d\mu_0+\int\varphi_1^{c_1} \d\mu_0 + \int(\varphi_0+ \varphi_1) \d\mu\\
&\qquad\ge\int \varphi_0^{c_0} \d\mu_0+ \int \varphi_1^{c_1}\d\mu_0+\int\min(\varphi_0+ \varphi_1,0) \d \mu\\
&\qquad\ge\int \varphi_0^{c_0} \d\mu_0+ \int \varphi_1^{c_1}\d\mu_0+\int\min(\varphi_0+ \varphi_1,0)\phi\,\d x
\end{split}\]
where we have used that $\mu \leq \phi$ in the last line. All the inequalities above should therefore be equalities which together with the continuity of $\varphi_0$ and $\varphi_1$ is easily seen to imply \eqref{cn1dc}-\eqref{cn2dc}-\eqref{cn3dc}. This shows the necessity of these conditions, the proof of sufficiency by duality is direct and therefore left to the reader. 
\end{proof}

\begin{corollary}\label{minimaloustidefi}
Under the same assumptions as in Proposition \ref{densityconstraintsol}, assume that $\mu$ is optimal for \eqref{eq:transportwcar} under the constraint $\mu\leq \phi$ and let $\gamma_0$ and $\gamma_1$ be optimal transport plans. Then, whenever $x_0,x,x_1$ are such that $(x_0,x)\in\spt(\gamma_0)$, $(x,x_1)\in\spt(\gamma_1)$, $x\in\spt(\phi-\mu)$, we have
$$c_0(x_0,x)+c_1(x,x_1)=\min_{y\in\spt(\phi-\mu)}\big\{c_0(x_0, y)+c_1(y, x_1)\big\}.$$
\end{corollary}

\begin{proof}
Let $(\varphi_0, \varphi_1)$ solve \eqref{dualdensityconstraint}. By construction, for every $(x_0, x_1, y)\in \spt(\mu_0)\times \spt(\mu_1)\times\spt(\phi)$ one has
$$c_0(x_0,y)+c_1(y,x_1)\ge\varphi_0^{c_0}(x_0)+\varphi_1^{c_1}(x_1)+(\varphi_0+\varphi_1)(y).$$
Together with \eqref{cn3dc} this implies that for every $(x_0,x_1)\in\spt(\mu_0)\times\spt(\mu_1)$
$$\min_{y\in\spt(\phi-\mu)}\{c_0(x_0,y)+c_1(y,x_1)\}\ge\varphi_0^{c_0}(x_0)+\varphi_1^{c_1}(x_1).$$
But now if $x\in\spt(\mu)\cap\spt(\phi-\mu)$, by \eqref{cn3dc} again we have $\varphi_0(x)+\varphi_1(x)=0$. Hence by \eqref{cn1dc}-\eqref{cn2dc} whenever $(x_0,x)\in\spt(\gamma_0)$, $(x,x_1)\in\spt(\gamma_1)$ and $x\in\spt(\phi-\mu)$ we have
\[\varphi_0^{c_0}(x_0)+ \varphi_1^{c_1}(x_1)=c_0(x_0, x)+c_1(x, x_1)\ge\min_{y\in \spt(\phi-\mu)} \{c_0(x_0,y)+c_1(y,x_1)\},\]
which yields the desired result.
\end{proof}

In the discrete case, we can easily deduce a bang-bang result stating that the constraint $\mu\leq \phi$ is always binding when $\mu>0$ under mild conditions on the cost. We will give similar bang-bang results for distance-like costs in Section \ref{sec:distance}.

\begin{corollary}\label{corodiscret}
Assume that $\mu_0$ and $\mu_1$ are discrete, that for every $(x_0, x_1)\in \spt(\mu_0)\times \spt(\mu_1)$, $c_0(x_0,.)$ and $c_1(., x_1)$ are $C^1$ and $M$-Lipschitz on $\spt(\phi)$ (for some $M$ that does not depend on $x_0$ and $x_1$) and that the set
\be\label{nullset}
\{x\in \spt(\phi)\; : \; \nabla_x c_0(x_0,x)+\nabla_x c_1(x, x_1)=0\}
\ee
is Lebesgue negligible. Then if $\mu$ is optimal for \eqref{eq:transportwcar} under the constraint $\mu\leq \phi$ there exists a measurable subset $E$ of $\spt(\phi)$ such that $\mu= \phi \mathds{1}_E$.
\end{corollary}

\begin{proof}
Let $(\varphi_0,\varphi_1)$ solve \eqref{dualdensityconstraint}, As seen in the proof of Proposition \ref{densityconstraintsol}, we may assume that, for every $x\in\spt(\phi)$
\[\begin{split}
&\varphi_0(x):=\min_{x_0 \in \spt(\mu_0)} \{c_0(x_0, x)-\varphi_0^{c_0}(x_0)\},\\
&\varphi_1(x):=\min_{x_1 \in \spt(\mu_1)} \{c_1(x, x_1)-\varphi_1^{c_1}(x_1)\},
\end{split}\]
so that $\varphi_0$ and $\varphi_1$ are Lipschitz hence differentiable a.e. on $\spt(\phi)$. Since $\varphi_0+ \varphi_1=0$ on $\spt(\mu)\cap \spt(\phi-\mu)$, we then have
\[ \nabla \varphi_0+ \nabla \varphi_1=0\qquad\mbox{ a.e. on }\{0<\mu <\phi\}\]
but if $\varphi_0$ (respectively $\varphi_1$) is differentiable at $x$ and $(x_0,x)\in\spt(\gamma_0)$ (resp. $(x, x_1)\in\spt(\gamma_1))$, where $\gamma_0$ and $\gamma_1$ are optimal plans, then
\[\nabla \varphi_0(x)=\nabla_x c_0(x_0,x),\qquad\nabla \varphi_1(x)=\nabla_x c_1(x,x_1).\]
Hence, denoting by $A_i$ the countable concentration set of $\mu_i$ ($i=0,1$), a.e. $x$ such that $0<\mu(x)<\phi(x)$ belongs to
\[\bigcup_{(x_0,x_1)\in A_0\times A_1}\{x\in\spt(\phi)\; :\;\nabla_x c_0(x_0,x)+\nabla_x c_1(x,x_1)=0\},\]
which is negligible by assumption. The desired bang-bang conclusion then readily follows.
\end{proof}

\begin{remark}
In some cases, for instance when the costs $c_0$ and $c_1$ depend quadratically or more generally as a $p$-th power of the distance (with $p>1$), the set in \eqref{nullset} reduces to a single point which depends in a Lipschitz way on $x_0$ and $x_1$. The conclusion of Corollary \ref{corodiscret} then still holds under the weaker assumption that one between $\mu_0$ and $\mu_1$ is discrete and the other one is singular with respect to the Lebesgue measure. More precisely, this still holds if the Hausdorff dimension of the support of $\mu_0$ is $h_0$, and the Hausdorff dimension of the support of $\mu_1$ is $h_1$, with $h_0+h_1<d$.
\end{remark}

\section{Distance like costs}\label{sec:distance}

In this section, we pay special attention to the case of distance-like costs:
\be\label{distcost}
c_0(x_0,x):=|x_0-x|^\alpha,\qquad c_1(x,x_1):=\lambda|x-x_1|^\alpha,
\ee
with $0<\alpha\le1$ and $\lambda>0$.

\subsection{Location constraint, concentration and integrability on the boundary}

Let us start with the case of a location constraint of type (ii): $\mu\in\cP(K)$ for some nonempty compact subset $K$ of $\R^d$.

\begin{lemma}\label{supportbyboundary}
Assume $K$ is a compact subset of $\R^d$ and that one of the following assumption holds:
\begin{itemize}
\item $\alpha=1$, $\lambda > 1$ and the interior of $K$ is disjoint from $\spt(\mu_1)$,
\item $\alpha\in(0,1)$ and the interior of $K$ is disjoint from $\spt(\mu_0)\cup\spt(\mu_1)$.
\end{itemize}
Then any solution $\mu$ of \eqref{eq:transportwcar} under the constraint $\mu\in\cP(K)$ is supported by $\partial K$.
\end{lemma}

\begin{proof}
For $(x_0,x_1)\in\spt(\mu_0)\times\spt(\mu_1)$, set
\[\begin{split}
&c(x_0,x_1):=\min_{x\in K}\{|x_0-x|^\alpha+\lambda|x-x_1|^\alpha\},\\
&T(x_0,x_1):=\argmin_{x\in K}\{|x_0-x|^\alpha+\lambda|x-x_1|^\alpha\}.
\end{split}\]
We know from Proposition \ref{prop:MultimargForm} that $\mu$ is supported by $T(\spt(\mu_0)\times \spt(\mu_1))$. In particular, if $x\in \spt(\mu)$ is an interior point of $K$ then it is a local minimizer of $c_0(x_0,\cdot)+c_1(\cdot,x_1)$ for some $(x_0, x_1)\in \spt(\mu_0)\times \spt(\mu_1)$. In the case $\alpha=1$, $\lambda > 1$, since $x\neq x_1$, this is clearly impossible. In the case $\alpha < 1$, our assumption implies that $x\notin\{x_0, x_1\}$, so that $x$ has to be a critical point of $ c_0(x_0,\cdot)+c_1(\cdot,x_1)$. One should have
$$\alpha|x-x_0|^{\alpha-2}(x-x_0) + \lambda\alpha|x-x_1|^{\alpha-2}(x-x_1) = 0,$$
so that $x_0\ne x_1$ and $x\in[x_0,x_1]$. But $c_0(x_0,\cdot)+c_1(\cdot,x_1)$ is strictly concave on $[x_0, x_1]$ which contradicts $x$ being a local minimizer.
\end{proof}

\begin{remark}
If $\alpha=\lambda=1$ the previous result is false: if $d=1$, $\mu_0=\delta_0$ and $\mu_1=\delta_1$, and $K=[1/4, 3/4]$, then it follows from the triangle inequality that any probability supported by $K$ is actually optimal.
\end{remark}

Now that we know that minimizers are supported by $\partial K$, one may wonder, if $K$ and $\mu_1$ are regular enough, whether these minimizers are absolutely continuous with respect to the $(d-1)$-Hausdorff measure on $\partial K$, the answer is positive if $\mu_0$ is discrete, i.e. is concentrated on a countable set, $\mu_0(K)=0$ and $\mu_1$ is absolutely continuous with support disjoint from $\iK$ (see Proposition \ref{dlikeinteg} below). A first step consists in the following result.

\begin{lemma}\label{lemmaTuniv}
Assume that $c_0$ and $c_1$ are as in \eqref{distcost} (with $\alpha\in(0,1]$ and $\lambda>1$ if $\alpha=1$), and that $K$ is compact. Then for every $x_0$ and (Lebesgue-)almost every $x_1\in\R^d\setminus K$, the set
$$T_{x_0}(x_1):=\argmin_{x\in K}\big\{|x_0-x|^\alpha+\lambda|x-x_1|^\alpha\big\}$$
is a singleton. 
\end{lemma}

\begin{proof}
Fix $x_0$, set
$$c_{x_0}(x_1):=\min_{x\in K}\big\{|x_0-x|^\alpha+\lambda|x-x_1|^\alpha\big\},$$
and observe that $c_{x_0}$ is locally Lipschitz on $\R^d\setminus K$. It thus follows from Rademacher's theorem that almost every $x_1\in\R^d\setminus K$ is a point of differentiability of $c_{x_0}$, and for such a point, if $x\in T_{x_0}(x_1)$, we have
\[\nabla c_{x_0}(x_1)=\lambda\alpha|x_1-x|^{\alpha-2}(x_1-x)\ne0.\]
If $\alpha\in (0,1)$ this immediately gives the claim with
\[T_{x_0}(x_1)=\{x_1+(\lambda\alpha)^{\frac{1}{1-\alpha}} |\nabla c_{x_0}(x_1)|^{\frac{2-\alpha}{\alpha-1}} \nabla c_{x_0}(x_1)\}.\]
If $\alpha=1$ and $\lambda>1$, if both $x$ and $x'$ belong to $T_{x_0}(x_1)$ then $x_1, x$ and $x'$ are aligned, so that the triangle inequality between their differences is saturated. But if $x\in [x_1, x')$, by the definition of $T_{x_0}(x_1)$ and $\lambda>1$, we should also have
\[\begin{split}
c_{x_0}(x_1)&=|x-x_0|+\lambda|x-x_1|=|x'-x_0|+\lambda|x'-x_1|\\
&=|x'-x_0|+\lambda(|x-x_1|+|x'-x|)\\
&>|x'-x_0|+|x'-x|+\lambda|x-x_1|,
\end{split}\]
which is impossible by the triangle inequality, yielding the a.e. single-valuedness of $T_{x_0}$ in this case as well.
\end{proof}

\begin{proposition}\label{dlikeinteg}
Assume that either $\alpha=1$, $\lambda>1$ or $\alpha\in (0,1)$ and
\begin{itemize}
\item $K$ is the closure of an open, bounded set in $\R^d$ with a boundary of class $C^{1,1},$
\item $\mu_0$ is discrete and $\mu_0(K) = 0$,
\item $\mu_1$ is absolutely continuous and $\iK \cap \spt \mu_1 = \emptyset$.
\end{itemize}
Then, any solution $\mu$ of \eqref{eq:transportwcar} under the constraint $\mu\in \cP(K)$ is absolutely continuous with respect to the $(d-1)$-Hausdorff measure on $\partial K$. 
\end{proposition}

\begin{proof}
Since $\mu_0$ is discrete, we can write $\mu_0 = \sum_{x_0 \in A_0} p_{x_0} \delta_{x_0},$ with $A_0$ at most countable, disjoint from $K$ and $p_{x_0}>0$ for every $x_0\in A_0$. It follows from Proposition \ref{prop:MultimargForm} and Lemma \ref{lemmaTuniv} that there exists a transport plan $\gamma$ between $\mu_0$ and $\mu_1$, which can be written as
$$\gamma=\sum_{x_0\in A_0} p_{x_0}\delta_{x_0} \otimes \mu_1^{x_0},$$
such that defining $T_{x_0}$ as in Lemma \ref{lemmaTuniv} and $T(x_0, x_1)=T_{x_0}(x_1)$ one has
\[\mu=T_\# \gamma=\sum_{x_0\in A_0} p_{x_0} {T_{x_0}}_\# \mu_1^{x_0}.\]
Since the second marginal of $\gamma$ is $\mu_1$, we also have 
\[\mu_1 =\sum_{x_0 \in A_0} p_{x_0}\mu_1^{x_0},\]
so that all the measures $\mu_1^{x_0}$ are dominated by $1/p_{x_0}\mu_1$ hence absolutely continuous. We are thus left to show that for each fixed $x_0$ in the countable set $A_0$, the measure ${T_{x_0}}_\# \mu_1^{x_0}$ (which is supported by $\partial K$ by Lemma \ref{supportbyboundary}) is absolutely continuous with respect to the $(d-1)$-Hausdorff measure on $\partial K$ which from now on we denote by $\sigma_{(d-1), \partial K}$. We now fix $x_0\in A_0$ and a Borel subset $A$ of $\partial K$ and our aim is to bound
\[({T_{x_0}}_\# \mu_1^{x_0})(A)=\mu_1^{x_0} (T_{x_0}^{-1}(A)).\]
To this end, let us distinguish the two cases $\alpha=1$, $\lambda>1$ and $\alpha\in (0,1)$.

\smallskip

Assume $\alpha=1$ and $\lambda>1$. Since $\mu_1(K)=0$ (because $\mu_1$ is absolutely continuous, $\partial K$ is a smooth hypersurface and thus Lebesgue negligible and $\mu_1(\iK)=0$), we have
$$\mu_1^{x_0} (T_{x_0}^{-1}(A)\setminus K) =\mu_1^{x_0} (T_{x_0}^{-1}(A)).$$
Now take $x=T_{x_0}(x_1)\in \partial K$ with $x_1\notin K$ which is $\mu_1$-a.e. the case (so that $x\notin \{x_0, x_1\}$). By optimality, there exists $\beta\ge0$ such that
\[\widehat{x-x_0}+\lambda \widehat{x-x_1}+\beta n(x)=0,\]
where for $\xi\in\R^d\setminus\{0\}$, we have set $\widehat{\xi}=\xi/|\xi|$, and
where $n(x)$ is the outward normal to $\partial K$ at $x$. Using the fact that $\lambda \widehat{x-x_1}$ has norm $\lambda$ yields
\[\lambda^2=\beta^2 +1+2\beta n(x)\cdot \widehat{x-x_0}\]
whose only nonnegative root is
\[\beta=\beta_{x_0}(x):=-n(x)\cdot\widehat{x-x_0}+\sqrt{\lambda^2-1+(n(x)\cdot\widehat{x-x_0})^2},\]
so that
\[\lambda\widehat{x_1-x}=\beta_{x_0}(x) n(x)+\widehat{x-x_0}\]
and the right hand side is a Lipschitz function of $x$ thanks to our assumptions ($\partial K$ being $C^{1,1}$ and $x_0$ being at a positive distance from $K$, hence from $x$). Using again that $\lambda \widehat{ x-x_1}$ has norm $\lambda$, this shows that if $x=T_{x_0}(x_1)$ then for some $r\in [0, R]$ with $R=\lambda^{-1}\diam(\spt\mu_1-K)$ there holds
\[x_1=F_{x_0}(r,x):=x+r[\beta_{x_0}(x) n(x)+\widehat{x-x_0}].\]
Hence
\[\mu_1^{x_0}(T_{x_0}^{-1}(A))\leq\mu_1^{x_0}(F_{x_0}([0,R]\times A)).\]
If $\sigma_{(d-1),\partial K}(A)=0$, the smoothness of $K$ and the fact that $F_{x_0}$ is Lipschitz on $[0,R]\times \partial K$, readily imply that $F_{x_0}([0,R]\times A)$ is Lebesgue negligible. Hence $\mu_1^{x_0}(T_{x_0}^{-1}(A))=0$ and since this holds for any $x_0\in A_0$, we also have $\mu(A)=0$, which implies the absolute continuity of $\mu$ with respect to $\sigma_{(d-1),\partial K}$. 

\smallskip 

Let us now assume that $\alpha\in (0,1)$. To cope with the fact that $c_1(x,x_1)$ is not differentiable if $x=x_1$, it will be convenient to fix $\eps>0$ and to consider $x_1\in A_1^\eps$, where 
\[A_1^\eps:=\{x_1\in\spt(\mu_1) \ ; \ \; d(K,x_1)\ge\eps\}\]
and
$$d(K,x):=\min_{y\in K}|x-y|$$
is the Euclidean distance to $K$. If $x_1\in A_1^\eps\cap T_{x_0}^{-1}(x)$ with $x\in A$, it follows from the first-order optimality condition, there is some $r\ge0$ such that
\[x_1=G_{x_0}(r,x) \eqset x+|H_{x_0}(r,x)|^{\frac{2-\alpha}{\alpha-1}} H_{x_0}(r, x),\]
where
\[H_{x_0}(r, x)=r n(x)+\lambda^{-1}|x-x_0|^{\alpha-2} (x-x_0).\]
Now, note that
\[|H_{x_0}(r,x)| = |x_1 - x|^{\alpha - 1}.\]
This shows that
\begin{align*}
|r|&\le|x_1-x|^{\alpha-1}+\lambda^{-1}|x-x_0|^{\alpha-1}\\
&\le\eps^{\alpha-1}+\lambda^{-1}\max_{x\in K}|x-x_0|^{\alpha-1}=:R_\eps(x_0).
\end{align*}
Hence, $A_1^\eps \cap T_{x_0}^{-1}(x)$ is included in the image by $G_{x_0}$ of the set $\{(r,x), \; x\in A, \ r\in [0, R_\eps(x_0)]\}$. Since $G_{x_0}$ is Lipschitz (with a Lipschitz constant depending on $\eps$) on this set we obtain as soon as $\sigma_{(d-1), \partial K}(A)=0$
\begin{align*}
\mu_1^{x_0}(T_{x_0}^{-1}(A)) &= \mu_1^{x_0}(T_{x_0}^{-1}(A)\setminus K)= \lim_{\eps \searrow 0 }\mu_1^{x_0}(T_{x_0}^{-1}(A)\cap A_1^\eps) \\
&\leq \lim_{\eps \searrow 0 }\mu_1^{x_0}(G_{x_0}([0, R_\eps(x_0)] \times A)) =0.
\end{align*}
Thus we can conclude as before that $\mu$ is absolutely continuous.
\end{proof}

\begin{proposition}\label{dlikeinteg2}
Suppose in addition to the assumptions of Proposition \ref{dlikeinteg} that $\mu_0$ has finite support, $\mu_1$ has a bounded density with respect to the $d$-dimensional Lebesgue measure. If $\alpha\in(0,1)$, further assume that $K\cap\spt\mu_1=\emptyset$. Then $\mu$ has a bounded density with respect to the $(d-1)$-Hausdorff measure on $\partial K$.
\end{proposition}

\begin{proof}
In the case $\alpha =1,$ $\lambda > 1$ we can continue using the same notation and Lipschitz mapping $F_{x_0}$ and $R$ as in the proof of Proposition \ref{dlikeinteg} to conclude for any Borel subset $A$ of $\partial K$
\begin{align*}
\mu(A)&=\sum_{x_0\in \spt(\mu_0)} p_{x_0}\mu_1^{x_0}(T_{x_0}^{-1}(A))\\
&\le \sum_{x_0\in\spt(\mu_0)} p_{x_0}\mu_1^{x_0} (F_{x_0}([0,R]\times A))\\
&\le \sum_{x_0\in\spt(\mu_0)} \|\mu_1\|_{L^\infty} \mathcal{L}^d(F_{x_0}([0,R]\times A))\\
&\le C{\mathrm{card}}(\spt \mu_0)\|\mu_1\|_{L^\infty} R\sigma_{(d-1),\partial K}(A)
\end{align*}
where $C$ is a constant that only depends on the $C^{1,1}$ smoothness of $\partial K$ and the maximal Lipschitz constant of $F_{x_0}$ over $[0,R]\times\partial K$, with respect to $x_0\in \spt(\mu_0)$. This way we deduce that $\mu\in L^\infty(\partial K,\sigma_{(d-1),\partial K})$.

For the case $\alpha\in(0,1)$ we need in addition $K\cap\spt\mu_1=\emptyset$ to ensure that, again using the same arguments and notation as in the proof of Proposition \ref{dlikeinteg}, there is an $\eps_0>0$ such that $A^{\eps_0}_1=\spt(\mu_1)$. In this way, all the analysis from the previous proof can be carried through on $A^{\eps_0}_1$ and we obtain
\begin{align*}
\mu(A)&=\sum_{x_0\in \spt(\mu_0)} p_{x_0} \mu_1^{x_0}(T_{x_0}^{-1}(A)\cap A^{\eps_0}_1)\\
&\le\sum_{x_0\in\spt(\mu_0)}p_{x_0}\mu_1^{x_0} (G_{x_0}([0,R_{\eps_0}(x_0)]\times A))\\
&\le\sum_{x_0\in\spt(\mu_0)}\|\mu_1\|_{L^\infty}\mathcal{L}^d(G_{x_0}([0,R_{\eps_0}(x_0)]\times A))\\
&\le C {\mathrm{card}}(\spt\mu_0)\|\mu_1\|_{L^\infty}R_{\eps_0}(x_0)\sigma_{(d-1),\partial K}(A),
\end{align*}
where $C$ is a constant that only depends on the $C^{1,1}$ smoothness of $\partial K$ and the maximal with respect to $x_0\in \spt(\mu_0)$ Lipschitz constant of $G_{x_0}$ over $[0,R_{\eps_0}(x_0)]\times \partial K$.
\end{proof}

One might also be interested in the case that the distribution of residents represented by $\mu_0$ is absolutely continuous and $\mu_1$ is discrete. The case $\alpha \in (0,1)$ is completely symmetric as we have not assumed $\lambda >1$ in the previous proofs. However for the case $\alpha = 1, \lambda >1$, the proof slightly differs as we shall see below. Arguing as for the proof of Lemma \ref{lemmaTuniv}, we have:

\begin{lemma}\label{lemmareverse}
Assume that $c_0$ and $c_1$ are as in \eqref{distcost} (with $\alpha\in(0,1]$ and $\lambda>1$ if $\alpha=1$), and that $K$ is compact. Then for (Lebesgue-)almost every $x_0\in\R^d\setminus K$ and every $x_1$, the set
$$T_{x_1}(x_0):=\argmin_{x\in K}\big\{|x_0-x|^\alpha+\lambda|x-x_1|^\alpha\big\}$$
is a singleton.
\end{lemma}

The analogue of Proposition \ref{dlikeinteg}, then reads 

\begin{proposition}
Assume that either $\alpha=1$, $\lambda>1$ or $\alpha\in (0,1)$ and
\begin{itemize}
\item $K$ is the closure of an open, bounded set in $\R^d$ with a boundary of class $C^{1,1},$
\item $\mu_0$ is absolutely continuous and $\iK\cap\spt\mu_0 = \emptyset,$
\item $\mu_1$ is discrete and $\mu_1(K) = 0.$
\end{itemize}
Then, any solution $\mu$ of \eqref{eq:transportwcar} under the constraint $\mu\in \cP(K)$ is absolutely continuous with respect to the $(d-1)$-Hausdorff measure on $\partial K$. 
\end{proposition}

\begin{proof}
As already explained, the case $\alpha>1$ can be handled exactly as for Proposition \ref{dlikeinteg}, we shall therefore assume that $\alpha=1$ and $\lambda>1$. We write $\mu_1 = \sum_{x_1 \in A_1} p_{x_1} \delta_{x_1}$, with $A_1$ countable and $p_{x_1}>0.$ It follows from Proposition \ref{prop:MultimargForm} and Lemma \ref{lemmareverse} that there exists a transport plan $\gamma$ between $\mu_0$ and $\mu_1$, which can be written as
$$\gamma=\sum_{x_1\in A_1} \mu_0^{x_1} \otimes p_{x_1}\delta_{x_1},$$
such that defining $T_{x_1}$ as in Lemma \ref{lemmaTuniv} and $T(x_0, x_1)=T_{x_1}(x_0)$ one has
\[\mu=T_\# \gamma=\sum_{x_1\in A_1} p_{x_1} {T_{x_1}}_\# \mu_0^{x_1}.\]
Since the first marginal of $\gamma$ is $\mu_0$, $\mu_0^{x_1}$ is absolutely continuous for every $x_1\in A_1$. We are thus left to show that for each fixed $x_1$ in the countable set $A_1$, the measure ${T_{x_1}}_\# \mu_0^{x_1}$ is absolutely continuous with respect to the $(d-1)$-Hausdorff measure on $\partial K$ which from now on we denote by $\sigma_{(d-1), \partial K}$. We now fix $x_1\in \spt(\mu_1)$ and a Borel subset $A$ of $\partial K$ and our aim is to bound
\[({T_{x_1}}_\# \mu_0^{x_1})(A)=\mu_0^{x_1} (T_{x_1}^{-1}(A)).\]
 Since $\mu_0(K)=0$, we have $\mu_0^{x_1} (T_{x_1}^{-1}(A)\setminus K) =\mu_0^{x_1} (T_{x_1}^{-1}(A))$. Now take $x=T_{x_1}(x_0)\in \partial K$ with $x_0\notin K$. By optimality, there exists $\beta\ge0$ such that
\[\widehat{x-x_0}+\lambda \widehat{x-x_1}+\beta n(x)=0 \mbox{ where for $\xi\in\R^d\setminus\{0\}$, we have set } \widehat{\xi}=\xi/|\xi|\]
where $n(x)$ is the outward normal to $\partial K$ at $x$. 
This time our aim is to write, for fixed $x_1$, $x_0$ as a Lipschitz function of $x$ and a length factor, so we proceed as follows.
Using the fact that $\lambda \widehat{x-x_1}$ has norm $\lambda$ yields
\[1=\beta^2 +\lambda^2+2\beta \lambda n(x)\cdot \widehat{x-x_1}.\]
This time, it is possible that there are two positive solutions for $\beta.$ We denote them
\begin{eqnarray*}
\beta^+_{x_1}(x) &\eqset -\lambda n(x) \cdot \widehat{x-x_1} + \sqrt{(\lambda n(x)\cdot\widehat{x-x_1})^2 + 1 - \lambda^2}, \\
\beta^-_{x_1}(x) &\eqset -\lambda n(x) \cdot \widehat{x-x_1} - \sqrt{(\lambda n(x)\cdot\widehat{x-x_1})^2 + 1 - \lambda^2}.
\end{eqnarray*}
Hence, we have one of the following equalities is satisfied by $(x_0,x,x_1)$
\begin{eqnarray*}
x_0 &= x + r(\lambda\widehat{x-x_1}+\beta^+_{x_1}(x)n(x)) =: F^+_{x_1}(r,x), \\
x_0 &= x + r(\lambda\widehat{x-x_1}+\beta^-_{x_1}(x)n(x)) =: F^-_{x_1}(r,x),
\end{eqnarray*}
where $r\in [0,R]$ and $R=\diam(\spt\mu_0-K)$.

Consider now a Borel set $A \subset \partial K$ with $\sigma_{(d-1), \partial K}(A) = 0.$
We distinguish the cases where the discriminant $(\lambda n(x)\cdot\widehat{x-x_1})^2 + 1 - \lambda^2$ is zero or positive
\begin{eqnarray*}
A_0 &\eqset &\left\{x \in A: (\lambda n(x)\cdot\widehat{x-x_1})^2+1-\lambda^2=0\right\},\\
A_{>} &\eqset &\left\{x \in A: (\lambda n(x)\cdot\widehat{x-x_1})^2+1-\lambda^2>0\right\}\\
&= &\bigcap_{\delta >0}\underbrace{\left\{x \in A: (\lambda n(x)\cdot\widehat{x-x_1})^2 + 1 - \lambda^2 \geq \delta\right\}}_{=: A_\delta}.
\end{eqnarray*}
Since $F_{x_1}^+$ and $F_{x_1}^-$ agree with Lipschitz functions on $[0,R]\times A_0$ and $[0,R]\times A_\delta$ for $\delta>0$ fixed, we obtain
\begin{align*}
\mu_0^{x_1}(T_{x_1}^{-1}(A))
&\le\mu_0^{x_1}(T_{x_1}^{-1}(A_0))+\lim_{\delta\searrow0}\mu_0^{x_1}(T_{x_1}^{-1}(A_\delta))\\
&\le\mu_0^{x_1}(F^+_{x_1}([0,R]\times A_0))+\lim_{\delta\searrow0}\left(\mu_0^{x_1}(F^+_{x_1}([0,R]\times A_\delta))+\mu_0^{x_1}(F^-_{x_1}([0,R]\times A_\delta))\right)\\
&=0,
\end{align*}
as required.
\end{proof}

It is unclear whether an $L^\infty$-bound can be obtained with the same proof strategy since the Lipschitz constant of the maps $F_{x_1}^+$ and $F_{x_1}^-$ may blow up as $\delta \rightarrow 0^+$. In addition, in Proposition \ref{dlikeinteg} the smoothness of $K$ is crucial, as the example below shows.

\begin{example}\label{edeltadist}
In the two-dimensional case take as $K$ the square $\{|x|+|y|\le1\}$ and consider the distance-like cost of Proposition \ref{dlikeinteg} with $\alpha=1$ and $\lambda>1$. Take as $\mu_0$ the Lebesgue measure on the disc $B(x_0,r)$ and as $\mu_1$ the Lebesgue measure on the disc $B(x_1,r)$, with $x_0=(-a,0)$ and $x_1=(a,0)$ as in Figure \ref{figdisc}. The optimal pivot measure $\mu$ has in this case a part proportional to the Dirac mass $\delta_{(1,0)}$ and in some cases, when $\lambda$ is large, $a$ is large, and $r$ is small, actually reduces to the Dirac mass $\delta_{(1,0)}$.

\begin{figure}[htbp]
\hskip-12cm
\begin{tikzpicture}
\tikz\draw[thick]
(-5,2) circle [radius=1cm] (-5,2) node {$\mu_0$}
(-2,2)--(0,4)--(2,2)--(0,0)--(-2,2) (0,2) node {K}
(5,2) circle [radius=1cm] (5,2) node {$\mu_1$};
\end{tikzpicture}
\caption{A nonsmooth constraint set $K$ may provide a singular optimal pivot measure.}\label{figdisc}
\end{figure}
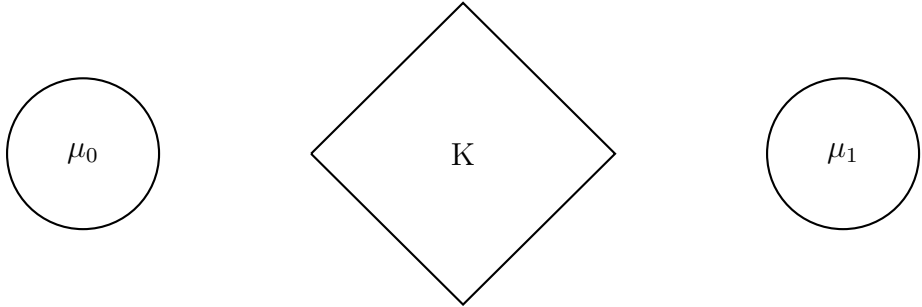
\end{example}

\subsection{Density constrained solutions are bang-bang}

We end this section by observing that in the case of a density constraint $\mu \leq \phi$, for distance-like costs minimizers are of bang bang type.

\begin{proposition}\label{bangbangdistlike}
Assume that $c_0$ and $c_1$ are as in \eqref{distcost} with $\lambda>1$ if $\alpha=1$, that $\phi\in L^1(\R^d)$ is nonnegative with compact support, that $\int\phi\,dx>1$, and that both $\spt(\phi)\cap \spt(\mu_0)$ and $\spt(\phi)\cap\spt(\mu_1)$ are Lebesgue negligible. Then any solution $\mu$ of \eqref{eq:transportwcar} under the constraint $\mu\le\phi$ is of the form $\mu=\phi\mathds{1}_E$ for some measurable subset $E$ of $\spt(\phi)$.
\end{proposition}

\begin{proof}
Let us start with the case $\alpha=1$, $\lambda>1$ and define $A:=\{0<\mu<\phi\}$, we then consider (Lipschitz) potentials $\varphi_0$ and $\varphi_1$ as in the proof of Corollary \ref{corodiscret}. A.e. point of $A$ is a differentiability point of $\varphi_0$ and $\varphi_1$, satisfies $\nabla \varphi_0+ \nabla \varphi_1=0$ and lies in $\R^d\setminus(\spt(\mu_0)\cup\spt(\mu_1))$. Hence arguing as in the proof of Corollary \ref{corodiscret}, for a.e. $x$ in $A$ one can find $x_0\in \spt(\mu_0)\setminus\{x\}$ and $x_1\in\spt(\mu_1)\setminus\{x\}$ such that
\[0=\nabla\varphi_0(x)+\nabla\varphi_1(x)=\frac{x-x_0}{|x-x_0|}+\lambda\frac{x-x_1}{|x-x_1|}\]
which is impossible since $\lambda>1$. This shows that $A$ is negligible and ends the proof for this case.

Consider now the slightly more complicated case where $\alpha \in (0,1)$, since $x\mapsto|x-x_0|^\alpha$ is Lipschitz only away from $x_0$, it is convenient for $\delta>0$ to introduce the set 
\[B_{\delta}:=\{x\in \spt(\phi) \; : \; d(x, \spt(\mu_0)\cup \spt(\mu_1))\geq \delta\}\]
on $B_\delta$ the potentials $\varphi_0$ and $\varphi_1$ are Lipschitz and we can find a subset $\tB_\delta$ of $B_{\delta}$ with $B_\delta \setminus \tB_\delta$ negligible such that $\varphi_0$ and $\varphi_1$ are differentiable on $\tB_\delta$. Consider now for $\eps>0$
\[A_\eps:=\{\eps <\mu <\phi-\eps\}\]
and let $\tA_\eps$ be the subset (of full Lebesgue measure in $A_\eps$ by Lebesgue's density Theorem) consisting of its points of density $1$, i.e. 
\[\tA_\eps:=\{y \in A_\eps \; : \; \lim_{r\to 0^+} \frac{ \mathcal{L}^d(  B(y,r)\cap A_\eps)  }{ \mathcal{L}^d( B(y,r))}=1\}.\]
Note that $\tA_\eps\subset\spt(\phi-\mu)$ 	and, arguing as before, for a.e. $x\in \tA_\eps \cap \tB_\delta$, we can find $(x_0, x_1)\in \spt(\mu_0)\times \spt(\mu_1)$ such that
\[\nabla \varphi_0(x)+\nabla\varphi_1(x)=\nabla f_{x_0, x_1}(x)=0,\]
where $f_{x_0,x_1}(x):=|x-x_0|^\alpha+\lambda|x-x_1|^\alpha$. Moreover we know from Corollary \ref{minimaloustidefi} that $\spt(\phi-\mu)$  is included in the level set $f_{x_0, x_1}\geq f_{x_0, x_1}(x)$ and so is  $A_\eps$, up to a Lebesgue negligible set, by continuity of $f_{x_0, x_1}$. Since $x \notin \{x_0, x_1\}$ is a critical point of $f_{x_0, x_1}$ we have $x_1\neq x_0$ and $x$ belongs to $[x_0, x_1]$, 
\[e:=\widehat{x-x_0}= \widehat{x_1-x}=\widehat{x_1-x_0}\]
and the  Hessian $D^2 f_{x_0, x_1}$ of $f_{x_0, x_1}$ at $x$ takes the form
\[D^2 f_{x_0,x_1}(x)=(\alpha|x-x_0|^{\alpha-2}+\lambda\alpha|x-x_0|^{\alpha-2})(\id+(\alpha-2)e\otimes e)\]
which shows that $x$ is a saddle-point of $f_{x_0, x_1}$, its hessian having a negative eigenvalue with eigenvector $e$ and being positive definite on $e^\perp$. Since for $y\in A_\eps$, we have 
\[f_{x_0,x_1}(y)=f_{x_0, x_1} (x)+ \frac{1}{2}  D^2 f_{x_0, x_1}(x)(y-x,y-x)+ o(\vert y-x\vert^2) \geq f_{x_0, x_1}(x) \]
we deduce that for $r>0$ small enough and some positive constant $\kappa$, whenever $y\in A_\eps \cap B(x,r)$, one has $y\in C_{x, e, \kappa}$ where
\[C_{x, e, \kappa}:=\{y \in \R^d \; : \;  \vert e \cdot (y-x)\vert \leq \kappa \vert (\id -e\otimes e) (y-x)\vert\}.\]
 Hence, for small $r>0$, $A_\eps \cap B(x,r) $ should lie inside  the strict cone $C_{x, e, \kappa}$ so that 
\[\limsup_{r\to 0^+} \frac{ \mathcal{L}^d(  B(x,r)\cap A_\eps)  }{ \mathcal{L}^d( B(x,r))} \leq \limsup_{r\to 0^+} \frac{ \mathcal{L}^d(  B(x,r)\cap C_{x, e, \kappa}  )  }{ \mathcal{L}^d( B(x,r))} <1\]
contradicting the fact that $x$ is a point of density $1$ of $A_\eps$.  This shows that $A_\eps\cap B_\delta$ is negligible, letting $\delta\to 0^+$ we find that $A_\eps$ is negligible and since this is true for every $\eps>0$, the desired conclusion follows.
\end{proof}

\section{The case of strictly convex costs with a convex location constraint}\label{sec:convex}

We now consider \eqref{eq:transportwcar} in the case of the location constraint $\cA=\cP(K)$ where $K$ is a compact convex subset of $\R^d$ with nonempty interior and $c_0$ and $c_1$ satisfy the strong convexity and smoothness assumptions:
\begin{equation}\label{sclike}
c_i (x,y):=F_i(y-x), \; F_i\in C^2(\R^d), \; \lambda \id \leq D^2 F_i \leq \Lambda \id, \; i=0, \; 1
\end{equation}
for some constants $0<\lambda \leq \Lambda$. Since these costs are twisted, \eqref{eq:transportwcar} in the case of the location constraint $\cA=\cP(K)$ admits a unique solution as soon as $\mu_0$ (or $\mu_1$) is absolutely continuous, see Remark \ref{rem:uniqu}. 

\begin{example}\label{edeltasquare}
Consider the two dimensional case with a location constraint given by the square $K$ of Example \ref{edeltadist}; take $\mu_0=\delta_{(-2,0)}$, $\mu_1$ uniform on the ball of radius $1$ centered at $(3,0)$, $c_0(x,y)=\vert x-y\vert^2$ and $c_1(x,y)=2\vert x-y\vert^2$. Then by a direct application of Proposition \ref{prop:MultimargForm}, the (unique) solution of \eqref{eq:transportwcar} is explicit: it is the image of the uniform measure on the ball $B$ of radius $2/3$ centered at $(4/3,0)$ by the projection onto $K$. It has an atom at $(1,0)$, an absolutely continuous part, uniform on $B\cap K$ and a one dimensional part corresponding to the points of $B$ which project onto the segments $[(0,1),(1,0)]$ and $[(0,-1), (1,0)]$.
\end{example}

This shows that, contrary to the case of distance like costs, one should expect that $\mu$ in general decomposes into a (nonzero) interior part and a boundary part:
\begin{equation}\label{decompmu}
\mu=\mui+\mub \mbox{ where } \mui(A):=\mu(A\cap \iK), \; \mub(A):=\mu(A \cap \partial K)
\end{equation}
for every Borel subset $A$ of $\R^d$. Regarding $\mub$, arguing as in Proposition \ref{dlikeinteg}, one can show that if $\mu_0$ is absolutely continuous, $\mu_1$ is discrete and $K$ is of class $C^{1,1}$, $\mub$ is absolutely continuous with respect to the $(d-1)$-Hausdorff measure on $\partial K$ (and has a bounded density if in addition $\mu_0 \in L^{\infty}$ and $\mu_1$ is finitely supported, see Proposition \ref{dlikeinteg2}). As for the regularity of $\mui$, we have:

\begin{proposition}
Assume $c_0$ and $c_1$ are of the form \eqref{sclike}, that $\mu_0$ and $\mu_1$ are compactly supported, with $\mu_0\in L^{\infty}$ and that $K$ is a compact convex subset of $\R^d$ with nonempty interior. Decomposing the solution $\mu$ of \eqref{eq:transportwcar} in the case of the location constraint $\cA:=\cP(K)$ as in \eqref{decompmu}, we have $\mui\in L^{\infty}$ and more precisely (identifying $\mui$ with its density), we have
\be\label{muiborne}
\|\mui\|_{L^\infty}\le\|\mu_0\|_{L^\infty} 2^d\lambda^{-d}\Lambda^d
\ee
where $\lambda$ and $\Lambda$ are the positive constants appearing in \eqref{sclike}.
\end{proposition}

To establish the $L^\infty$ bound in \eqref{muiborne}, we shall use a penalization strategy, detailed in the next paragraph, the proof by a standard $\Gamma$-convergence argument is postponed to the end of this section.

\subsection{Penalization}

Given $g\in C^2(\R^d)$, with $g$ convex and nonnegative, let us consider
\begin{equation}\label{penalizedpbm}
\inf_{\mu \in \cP(\R^d)} T(\mu)+ \int_{\R^d} g \mu \mbox{ with } T(\mu):=W_{c_0}(\mu_0, \mu)+ W_{c_1}(\mu_1, \mu)
\end{equation}
then we have:

\begin{proposition}
Assuming \eqref{sclike} and $\mu_0\in L^\infty$, \eqref{penalizedpbm} admits a unique solution $\mu_g$. Moreover $\mu_g$ is absolutely continuous with respect to the Lebesgue measure and its density (still denoted $\mu_g$) satisfies for a.e. $x\in \R^d$, the bound
\be\label{boundmug}
\mu_g(x)\le\|\mu_0\|_{L^\infty}\lambda^{-d}\det(D^2 g(x)+2\Lambda\id)
\ee
where $\lambda$ and $\Lambda$ are the positive constants appearing in \eqref{sclike}.
\end{proposition}

\begin{proof}
The coercivity of $c_0$, $c_1$ and $g\ge0$ easily give the existence of a minimizer as in Proposition \ref{prop:MultimargForm} (incorporating $g$ in one of the costs considered there), whereas uniqueness is guaranteed by twistedness of the costs and the absolute continuity of $\mu_0$, see Remark \ref{rem:uniqu}. Also Proposition \ref{prop:MultimargForm} ensures there is some ball $B$ which contains a neighbourhood of $\spt(\mu_g)$. Then, Theorem 3.3 from Pass \cite{Pass} guarantees that the minimizer $\mu_g$ is absolutely continuous. The optimality condition derived from the dual formulation of \eqref{penalizedpbm}, (see \eqref{eq:DualCtrans}) gives the existence of potentials $\varphi_0$ and $\varphi_1$ such that 
\begin{equation}\label{oc00}
\varphi_0 + \varphi_1 +g=0 \mbox{ on $B$}
\end{equation}
and
\[W_{c_0}(\mu_0, \mu_g)=\int_{\R^d} \varphi_0^{c_0} \mu_0 + \int_{\R^d} \varphi_0 \mu_g, \; W_{c_1}(\mu_g, \mu_1)=\int_{\R^d} \varphi_1^{c_1} \mu_1 + \int_{\R^d} \varphi_1 \mu_g\]
so that defining the $c_i$-concave potentials
\[\begin{split}
&\tfi_0(x):=\inf_{x_0\in\spt(\mu_0)}\{c_0(x_0,x)-\varphi_0^{c_0}(x_0)\},\\
&\tfi_1(x):=\inf_{x_1\in\spt(\mu_1)}\{c_1(x,x_1)-\varphi_1^{c_1}(x)\},
\end{split}\]
one should have
\be\label{coinct}
\varphi_i\le\tfi_i\mbox{ on $B$\qquad and\qquad}\varphi_i=\tfi_i\mbox{ on $\spt(\mu_g)$}.
\ee
Now observe that thanks to \eqref{sclike}, $\tfi_0$ and $\tfi_1$ are semi-concave and more precisely
\be\label{semiconc}
D^2\tfi_i\le\Lambda\id,\qquad i=0, 1.
\ee
In particular $\tfi_0$ and $\tfi_1$ are everywhere superdifferentiable, but on $\spt(\mu_g)$, thanks to \eqref{oc00} and \eqref{coinct}, $\tfi_0 + \tfi_1+g$ is minimal and since $g$ is differentiable this implies that $\tfi_0+ \tfi_1$ is also subdifferentiable on $\spt(\mu_g)$. This readily implies that $\tfi_0$ and $\tfi_1$ are differentiable on $\spt(\mu_g)$ and that
$$\nabla \tfi_0+ \nabla \tfi_1 + \nabla g =0 \mbox{ on $\spt(\mu_g)$}.$$
The functions $\tfi_0$ and $\tfi_1$ are semi-concave and, by Alexandrov's Theorem (see Theorem 6.9 in \cite{EG}), they are twice differentiable $\mu_g$-a.e.; the minimality of $\tfi_0+\tfi_1+g$ on $\spt(\mu_g)$ then gives
\be\label{Hesstfp}
D^2 \tfi_0 + D^2 \tfi_1 + D^2 g \ge0 \mbox{ $\mu_g$-a.e.}.
\ee
The optimal transport $S_0$ for the cost $c_0$ between $\mu_g$ and $\mu_0$ (see Theorem 3.7 in \cite{GMC}) is then given by
\[S_0(x)=x- \nabla F_0^*(\nabla \tfi_0(x)), \; x\in \spt(\mu_g),\]
where $F_0^*$ is the Legendre transform of $F_0$. The absolute continuity of $\mu_g$ enables us to use Theorem 4.8 of Cordero-Erausquin \cite{Cordero} to get the existence of a set of full measure for $\mu_g$ for which one has the Jacobian equation
\be\label{jacobian00}
\mu_g=\mu_0 \circ S_0\det(\id-D^2 F_0^*(\nabla \tfi_0) D^2 \tfi_0),
\ee
where $D^2 \tfi_0(x)$ is to be understood in the sense of Alexandrov and the matrix $\id-D^2 F_0^*(\nabla \tfi_0) D^2 \tfi_0$ which is diagonalizable with real and nonnegative eigenvalues can be rewritten as 
\[\id-D^2 F_0^*(\nabla \tfi_0) D^2 \tfi_0 = D^2 F_0^*(\nabla \tfi_0) (D^2 F_0(x-S_0(x)) -D^2 \tfi_0(x)).\]
Together with \eqref{jacobian00}, since $D^2 F_0^* \leq \lambda^{-1}\id$ and $D^2 F_0(x-S_0(x)) -D^2 \tfi_0(x)$ is semidefinite positive, this gives for $\mu_g$ a.e. $x$:
\[\mu_g(x)\le\|\mu_0\|_{L^\infty}\lambda^{-d} \det(D^2 F_0(x-S_0(x))-D^2\tfi_0(x))\]
by \eqref{Hesstfp} and \eqref{semiconc}, we then have
\[-D^2\tfi_0(x)\le D^2g(x)+D^2\tfi_1(x)\le D^2g(x)+\Lambda\id\]
but since $D^2 F_0\le\Lambda\id$, the bound \eqref{boundmug} follows.
\end{proof}

\subsection{Proof of the bound by $\Gamma$-convergence}

Recall that we have assumed that $K$ is a convex compact subset with nonempty interior, for $\eps>0$, setting $K_\eps:=K+\eps B$ (where $B$ is the unit Euclidean ball of $\R^d$); consider the mollifiers $\eta_\eps=\eps^{-d}\eta(\frac{\cdot}{\eps})$ with $\eta$ a smooth probability density supported on $B$, consider the smooth and convex function
\[g_\eps:= \eta_\eps \star \eps^{-1} d_{K_\eps}^2\]
where $d_{K_\eps}$ is the distance to $K_\eps$. Defining $T$ as in \eqref{penalizedpbm} and for every $\nu \in \cP(\R^d)$:
\[J_\eps(\nu):=T(\nu)+ \int_{\R^d} g_\eps \nu, \; J(\nu):=\begin{cases} T(\nu) \mbox{ if $\nu \in \cP(K)$} \\ + \infty \mbox{ otherwise} \end{cases}\]
it is easy to see that $J_\eps$ $\Gamma$-converges to $J$ as $\eps \to 0^+$ for the narrow topology. Hence the tight sequence of minimizers of $J_\eps$, $\mu_\eps:=\mu_{g_\eps}$ converges narrowly to $\mu$ the minimizer of $J$ i.e. the solution of \eqref{eq:transportwcar} with $\cA=\cP(K)$. Since $D^2 g_\eps=0$ on $\iK$, we deduce from \eqref{boundmug} that for every open $\Omega$ such that $\Omega \Subset \iK$
\[\|\mu_\eps\|_{L^\infty(\Omega)}\le\|\mu_0\|_{L^\infty} 2^d\lambda^{-d}\Lambda^d\]
from which one deduces \eqref{muiborne} by letting $\eps\to0^+$.

\section{A parking location model}\label{sec:parking}

In this section, we introduce a mathematical model for the optimal location of a parking area in a city. We fix:
\begin{itemize}
\item a compactly supported probability measure on $\R^d$, $\nu_0$ which represents the distribution of residents in a given area; 
\item a a compactly supported probability measure on $\R^d$, $\nu_1$ which represents the distribution of services.
\end{itemize}
The goal is to determine a measure $\mu$ which represents the density of parking places, in order to minimize a suitable total transportation cost. All the residents travel to reach the services, but some of them may simply walk (which will cost $c_1(x,y)$ to go from $x$ to $y$), while some other ones may use their car to reach a parking place (which will cost $c_0(x,y)$ to go from $x$ to the parking place $y$) and then walk from the parking place to the services (which will cost $c_1(y,z)$ to go from $y$ to $z$). We consider two cost functions $c_0$ and $c_1$ and the corresponding Monge-Kantorovich functionals $W_{c_0}$ and $W_{c_1}$, defined as in \eqref{defwasser}, respectively representing the cost of moving by car and the cost of walking. It may be natural to assume that walking is more costly than driving i.e. $c_1\ge c_0$, for instance we may take $p>0$ and
\be\label{parkc}
c_0(x,y)=|x-y|^p,\qquad c_1(x,y)=\lambda|x-y|^p\hbox{ with }\lambda\ge1.
\ee
Assuming that $\mu_0\le\nu_0$ denotes the distribution of driving residents and $\mu_1\le\nu_1$ the corresponding services they reach for, the total cost we consider is
\be\label{parkF}
F(\mu_0,\mu_1,\mu)=W_{c_1}(\nu_0-\mu_0,\nu_1-\mu_1)+W_{c_0}(\mu_0,\mu)+W_{c_1}(\mu,\mu_1).
\ee
The optimization problems we consider are then the minimization of $F(\mu_0,\mu_1,\mu)$, subject to the constraints
\[0\le\mu_0\le\nu_0,\quad 0\le\mu_1\le\nu_1,\quad \int d\mu_0=\int d\mu_1=\int d\mu,\]
and additional constraints as:
\begin{itemize}
\item no other constraints on the parking density $\mu$;
\item location constraints, that is $\spt\mu\subset K$, with a compact set $K\subset\R^d$ a priori given;
\item density constraints, that is $\mu\le\phi$, for a given nonnegative and integrable function $\phi$.
\end{itemize}
This optimization problem in the case of a location constraint can also be reformulated as a linear program in the following way
\be\label{eq:displacementparking}
\inf \int_{\R^d\times\R^d} c_1(x_0,x_1)\d \gamma(x_0,x_1) + \int_{\R^d \times \R^d \times K} (c_0(x_0,x)+c_1(x,x_1)) \d \tilde\gamma(x_0,x,x_1)
\ee
subject to the constraints:
\[ \gamma, \tilde \gamma\ge0, \; \gamma+{\pi_{0,1}}_\#\tilde\gamma\in\Pi(\nu_0,\nu_1).\]
It is indeed easy to see that the optimal solution to minimizing the functional in \eqref{parkF} is given by ${\pi_{\piv}}_\# \tilde \gamma$. Hence to incorporate a density constraint in the formulation \eqref{eq:displacementparking} one needs to add the constraint ${\pi_{\piv}}_\# \tilde \gamma \leq \phi.$
The problem with location constraint is actually equivalent to a standard optimal transport problem with cost function
\begin{equation*}
C(x_0,x_1) \eqset \min\left\{c_1(x_0, x_1), \inf_{x \in K}\{c_0(x_0,x) + c_1(x,x_1)\} \right\}.
\end{equation*}
More precisely, consider 
\begin{equation}\label{eq:parkingOneCost}
\inf_{\beta \in \Pi(\nu_0, \nu_1)}\int_{\R^d \times \R^d} C(x_0,x_1) \d \beta (x_0,x_1).
\end{equation}
Then both \eqref{eq:displacementparking} and \eqref{eq:parkingOneCost} admit solutions and they are equivalent in the following sense
\begin{itemize}
\item $\min \eqref{eq:displacementparking} = \min \eqref{eq:parkingOneCost},$
\item if $\gamma, \tilde\gamma$ are optimal for \eqref{eq:displacementparking}, then $\beta \eqset \gamma+{\pi_{0,1}}_\#\tilde\gamma$ is optimal for $\eqref{eq:parkingOneCost},$
\item if $\beta$ is optimal for \eqref{eq:parkingOneCost}, then defining
$$V_1 \eqset \left\{ (x_0,x_1) \in \R^d \times \R^d:~ c_1(x_0,x_1) = C(x_0,x_1)\right\}$$
$\gamma \eqset \beta_{|V_1}$ and $P:\R^d \times \R^d \rightarrow \R^d$ (measurable)
$$P(x_0,x_1) \in \argmin_{x \in K} \{c_0(x_0,x) + c_1(x,x_1)\},$$
and
$$\d\tilde\gamma(x_0,x_1,x) \eqset \delta_{P(x_0,x_1)}(x) \otimes \d\beta_{|\R^d \setminus V_1}(x_0,x_1),$$
then $\gamma$ and $\tilde \gamma$ are optimal for \eqref{eq:displacementparking}.
\end{itemize}

\begin{remark} 
Note that the solutions $(\mu_0,\mu,\mu_1)$ to minimizing \eqref{parkF}, (respectively the solutions $\gamma$ and $\tilde \gamma$ to \eqref{eq:displacementparking}) are not necessarily probability measures. The optimal common total mass of $\tilde\gamma$, $\mu_0$ and $\mu_1$ represents the fraction of $\nu_0$ which uses the parking. Thus, the parking problem is a generalization of the interpolation problem from Section \ref{sec:wasser}, which corresponds to imposing that the parking measure is of full mass.
\end{remark}

However, under quite general and natural assumptions, it can be shown that the optimal parking measure is non-trivial

\begin{lemma}\label{lemmunontrivial}
Assume that $c_0, ~ c_1 \geq 0$ are continuous with $c_0(x,x)=c_1(x,x)=0$ and $c_1(x,y)>c_0(x,y)$ for $x\neq y \in \R^d$. Consider the density constraint case $\mu \leq \phi$ with $\phi \in L^1(\R^d)$, $0<\phi<+\infty$ almost everywhere. Then, if $\nu_0 \neq \nu_1$ the optimal $\mu$ for the parking problem is non-trivial, i.e. $\mu \neq 0.$
\end{lemma}
\begin{proof}
Assume by contradiction that $\mu=0$ is an optimal solution. The optimal cost for the parking problem is then given by
\begin{equation}
W_{c_1}(\nu_0,\nu_1).
\end{equation}
Let $\gamma_1$ be an associated optimal transport plan. Since $\nu_0 \neq \nu_1$, there is $(x_0,x_1) \in \spt \gamma_1$ with $x_0 \neq x_1$. Clearly, there exists $x \in \R^d$ (take for instance $x=x_1$) such that
\begin{equation}
c_0(x_0,x)+ c_1(x,x_1) < c_1(x_0,x_1).
\end{equation}
Then by continuity of the cost functions there exists an open neighborhood of the form $A_0 \times A \times A_1$ of $(x_0,x,x_1)$ such that the inequality remains valid, i.e.
\begin{equation}
c_0(y_0,y)+ c_1(y,y_1) < c_1(y_0,y_1),
\end{equation}
for all $(y_0,y,y_1) \in A_0 \times A \times A_1$.
 By possibly choosing a smaller (open) $A$ we can assume that
$$
\int_{A} \phi(x) \d x \leq \gamma_1(A_0 \times A_1),
$$
so that there is $t \in (0,1]$ such that
$$
\int_{A} \phi(x) \d x = t \gamma_1(A_0 \times A_1).
$$
But then
\begin{align}
 W_{c_1}(\nu_0,\nu_1) &= \int c_1(y_0,y_1) \d \gamma_1(y_0,y_1) \\
&> \int_{(A_0 \times A_1)^c} c_1(y_0,y_1) \d \gamma_1(y_0,y_1) + (1-t)\int_{A_0 \times A_1} c_1(y_0,y_1) \d \gamma_1(y_0,y_1)\\
& \quad + (\gamma_1(A_0 \times A_1))^{-1} \int_{A_0 \times A_1}\int_A (c_0(y_0,y) +c_1(y,y_1))\phi(y)\d y \d \gamma_1(y_0,y_1) \\
&\geq W_{c_1}(\nu_0 - \mu_0, \nu_1 - \mu_1) + W_{c_0}(\mu_0,\mu') + W_{c_1}(\mu',\mu_1)
\end{align}
where $\mu'=\phi_{|A}$, and $\mu_0$, $\mu_1$ are the marginals of $t {\gamma_1}_{| A_0 \times A_1}$. This gives a contradiction and achieves the proof.
\end{proof}

\subsection{Examples}

We first solve a simple particular example in $\R^2$ before giving some numerical simulations. This example shows that in some cases the optimal choices for $\mu_0,\mu_1,\mu$ are not of unitary mass, that corresponds to the cases where it is more efficient for some residents to walk from their residence to the services without using their car.
\begin{example}\label{parkingexample}
Let $\nu_0=\delta_{x_0}$ and $\nu_1=\delta_{0}$ be two Dirac masses in $\R^2$; with $x_0\neq 0$. We consider the costs $c_0$ and $c_1$ as in \eqref{parkc} with $p> 0$ and $\lambda> 1$. Then $\mu_0=\alpha\delta_{x_0}$ and $\mu_1=\alpha\delta_0$, for some $\alpha\in [0,1]$, and the optimization problems for the functional $F$ in \eqref{parkF} become the minimization of the quantity
\[\begin{split}
\lambda(1-\alpha)|x_0|^p&+\int\big(|x-x_0|^p+\lambda|x|^p\big)\,d\mu\\
&=\lambda|x_0|^p+\int\big(|x-x_0|^p+\lambda|x|^p-\lambda|x_0|^p\big)\,d\mu.
\end{split}\]
Since $\lambda|x_0|^p$ is fixed we are reduced to minimize the quantity
$$\int\big(|x-x_0|^p+\lambda|x|^p-\lambda|x_0|^p\big)\,d\mu$$
with the constraint $\int d\mu\le1$ and possibly other location and density constraints on $\mu$, as illustrated above. Setting
\be\label{parkf}
f(x)=|x-x_0|^p+\lambda|x|^p-\lambda|x_0|^p
\ee
it is clear that $\mu$ has to be concentrated on the set where $f\le0$. The optimization problem with no other constraints on $\mu$ has then the trivial solution $\alpha=1$ and $\mu=\delta_{\argmin f}$ (for instance $\mu=\delta_0$ if $p=1$ and $\mu=\delta_{(1+\lambda)^{-1} x_0}$ if $p=2$). 
The situation becomes more interesting when other constraints on $\mu$ are present. If we impose $\spt\mu\subset K$, let $\bar x\in K$ be a minimum point of the function $f$ in \eqref{parkf} over $K$. If $f(\bar x)<0$ then $\alpha=1$ and $\mu=\delta_{\bar x}$ is a solution; if $f(\bar x)\ge0$ then $\alpha=0$ and $\mu=0$ is a solution.

We consider now the more realistic case when a density constraint on $\mu$ is imposed, we take $\mu\le1$. The optimal measure $\mu$ for the cost \eqref{parkF} is then the characteristic function $1_{A_c}$ of a suitable level set $A_c=\{f\le c\}$ with $c\le0$. Thus the following situations may occur.
\begin{itemize}
\item If $|\{f\le0\}|\ge1$ then $\alpha=1$ and $\mu=1_{A_c}$, where the level $c\le0$ is such that $|A_c|=1$. Note that, since the function $f$ is convex, the set $A_c$ is convex too. This happens when $x_0$ is far enough from the origin, and all people then drive to the parking area $A_c$.
\item If $|\{f\le0\}|<1$ then $\alpha=|\{f\le0\}|$ and $\mu=1_{A_0}$.
\end{itemize}

For instance, when $p=2$ it is easy to see that the set $A_0$ is the ball centered at $x_0/(1+\lambda)$ with radius $\lambda|x_0|/(1+\lambda)$. Therefore:
\begin{itemize}
\item if $|x_0|\ge\pi^{-1/2}(\lambda+1)/\lambda$ we have $\alpha=1$ and $\mu_{opt}=1_A$, where $A$ is the disk centered at $x_0/(1+\lambda)$ of unitary area;
\item if $|x_0|<\pi^{-1/2}(\lambda+1)/\lambda$ we have $\alpha=\pi|x_0|^2\lambda^2/(\lambda+1)^2$ and $\mu_{opt}=1_{A_0}$. In this case only the fraction $\alpha$ of people drive to reach a parking, while the rest of residents walk up to the services.
\end{itemize}
In Figure \ref{fig1} the two situations are graphically represented in the cases $p=\lambda=2$, while in Figure \ref{fig2} we plot the two optimal parking areas when $p=1$ and $\lambda=2$.

\begin{figure}[h!]
\centering
{\includegraphics[scale=0.7]{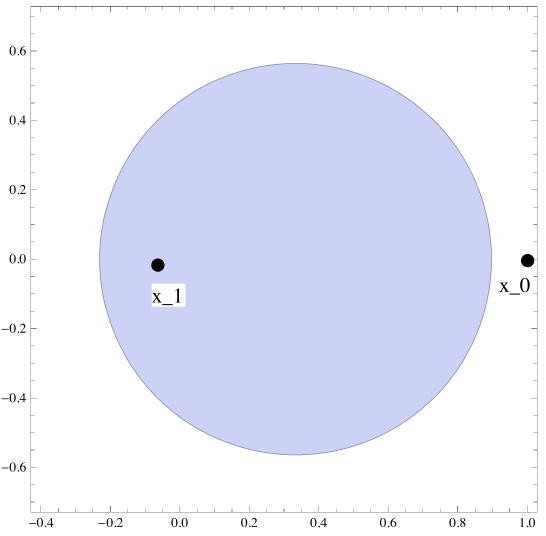}
\includegraphics[scale=0.7]{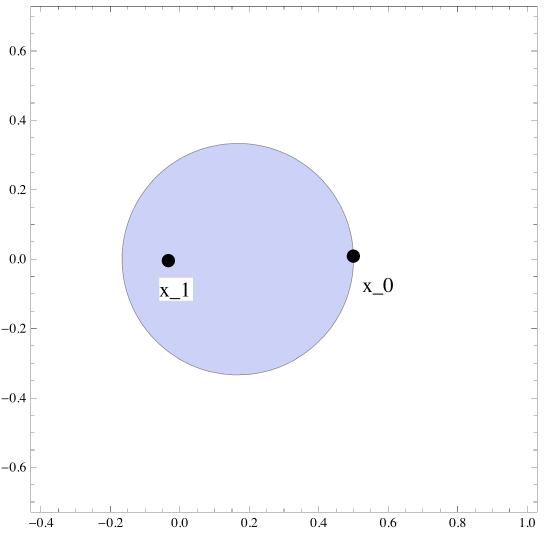}}
\caption{The case $p=\lambda=2$. On the left $|x_0|=1$ gives $|A_{opt}|=1$; on the right $|x_0|=1/2$ gives $|A_{opt}|\simeq0.35$.}\label{fig1}
\end{figure}

\begin{figure}[h!]
\centering
{\includegraphics[scale=0.7]{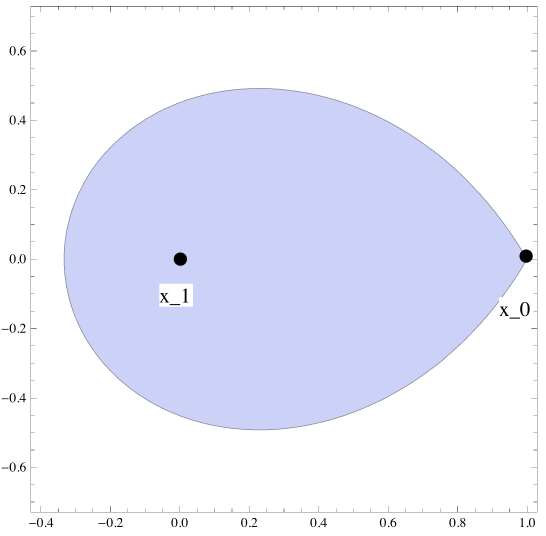}
\includegraphics[scale=0.7]{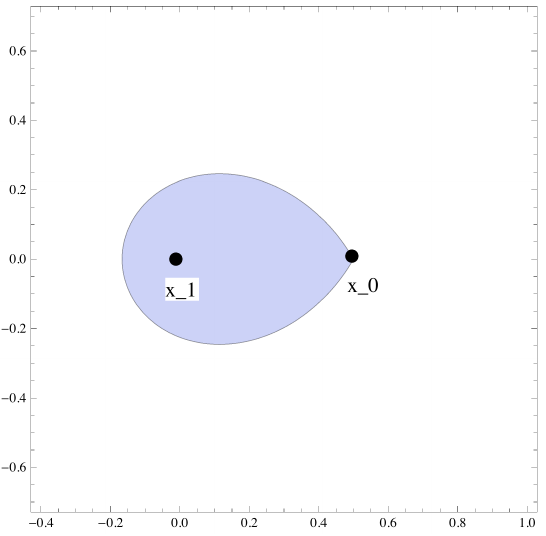}}
\caption{The case $p=1$ and $\lambda=2$. On the left $|x_0|=1$ gives $|A_{opt}|\simeq0.97$; on the right $|x_0|=1/2$ gives $|A_{opt}|\simeq0.24$.}\label{fig2}
\end{figure}
\end{example}

\section{Numerical simulations}\label{sec:numeric}
For the numerical simulation of examples in the case of interpolation between measures \eqref{eq:transportwcar} and the parking problem \eqref{parkF} we replace the optimal transportation costs by their entropically regularized versions. This will enable us to apply some variants of the celebrated Sinkhorn's algorithm, popularized in the context of optimal transport and matching by \cite{cuturi_sinkhorn_2013} and \cite{galichon_matching_2010}. For an introduction to this rapidly developing subject and convergence results, we refer the reader to \cite{peyre_computational_2019} and \cite{nutz_introduction_nodate}.

\subsection{Description of the Sinkhorn-like algorithm}
The entropically regularized optimal transport cost for a cost function $c$, a regularizing parameter $\eps >0$ and a fixed reference measure $Q \in \cP(\R^d \times \R^d)$ is given by
\be\label{entregc}
\inf\left\{\int_{\R^d\times\R^d} c(x,y)\,\d\gamma(x,y)+\eps H(\gamma|Q)\; :\ \gamma\in\Pi(\mu_0,\mu_1)\right\},
\ee
where the relative entropy $H(P|Q)$ between two nonnegative finite measures $P, Q$ on $\R^d$ is defined by
$$H(P|Q)\eqset\begin{cases}
\int_{\R^d} \left(\log\left(\frac{\d P}{\d Q}\right) -1\right) \d P \mbox{ if $P\ll Q$,}\\
+\infty\mbox{ otherwise.}
\end{cases}$$
Note that, by setting $R =e^{-c/\eps}Q$, we have
$$\eps H(\gamma | R) = \int_{\R^d\times\R^d} c(x,y)\,\d\gamma(x,y)\ + \eps H(\gamma|Q)$$
so that \eqref{entregc} amounts to minimizing $H(.|R)$ among transport plans between $\mu_0$ and $\mu_1$. As already observed in \cite{benamou_iterative_2015} in Section 3.2 the entropically regularized version of \eqref{eq:transportwcar} becomes for two suitably chosen reference measures $R_0$, $R_1$
\be\label{eq:EntregInt}
\inf\Big\{ H(\gamma_0|R_0)+ H(\gamma_1|R_1)\ :\ \mu\in\cA,\ \gamma_0\in\Pi(\mu_0,\mu),\ \gamma_1\in\Pi(\mu,\mu_1)\Big\}.
\ee
The cases \ref{case:Noconstraints} with no additional constraint and \ref{case:Location} with location constraint $K$ can be treated by choosing the reference measures to enforce the support of $\mu$ being included in $K$. Namely we choose 
\be\label{eq:choiceRefmeasint}
R_0 =e^{-c_0/\eps}\mu_0\otimes\mathds{1}_{K},\quad R_1 =e^{-c_1/\eps}\mathds{1}_{K}\otimes\mu_1,
\ee
where for case \ref{case:Noconstraints} we choose $K$ large enough (yet still compact) as before. The resulting Sinkhorn iterations are standard, see for instance Proposition 1 and 2 in \cite{benamou_iterative_2015}.
The case \ref{case:Densityconstr} of a density constraint $\phi$ requires performing a suitable projection of the estimated interpolation, as specified in Proposition 4.1 in \cite{peyre_entropic_2015} in the case of $\phi \equiv \kappa.$ We write the corresponding Sinkhorn iterations including the projection for the density constraint for sake of completeness in its dual form where the algorithm essentially becomes alternate gradient ascent. For this, note that the dual of \eqref{eq:EntregInt} with $R_i$ as in \eqref{eq:choiceRefmeasint} in the case of a density constraint $\mu \leq \phi$ ($\spt \phi \subset K$) is given by
\[
\sup_{\substack{\varphi_0,\varphi_1,\\ \psi_0,\psi_1}} \left\{ - \sum_{i=0}^1 \int_{\R^{2d}} \exp\left(\varphi_i + \psi_i \right)\d R_i+ \sum_{i=0}^1 \int_{\R^d} \varphi_i \d \mu_i + \int_{\R^d} \left( \psi_0 + \psi_1\right) \phi \d x : \psi_0 + \psi_1 \leq 0\right\}.
\]
Sinkhorn iterations are given by the explicit coordinate ascent updates for this dual formulation:
\[\begin{split}
&\exp\left(\varphi_i^{l+1}(x_i)\right)=\left( \int_{K} \exp\left(-\frac{c_i}{\eps} + \psi_i^l(x) \right) \d x \right)^{-1},\\
&\exp\left(\psi_i^{l+1}(x)\right)=\min\left\{\mu^l,\phi\right\} \left( \int_{\R^d} \exp\left(-\frac{c_i}{\eps} + \varphi_{i}^{l+1}(x_i) \right) \d \mu_i (x_i) \right)^{-1},
\end{split}\]
where $\mu^l$ is the current approximate interpolation which is given by the geometric mean formula (see Proposition 2 of \cite{benamou_iterative_2015}):
$$\mu^l = \prod_{j=0}^1 \left(\int \exp\left(-\frac{c_j}{\eps} + \varphi_j^{l+1} + \psi_j^l \right) \d \mu_j (x_j)\right)^{\frac{1}{2}}.$$
Regularizing the parking problem \eqref{parkF} in a similar way leads to
\be\label{eq:parkingentrreg}
\inf_{\gamma,\tilde{\gamma_0},\tilde{\gamma_1},\tilde{\gamma}\in\cM} \left\{H\left(\gamma|R\right)+H\left(\tilde{\gamma_0}|R_0\right)+H\left(\tilde{\gamma_1}|R_1\right) \right\},
\ee
where
\[\begin{split}
\cM&=\Big\{\gamma,\tilde{\gamma_0},\tilde{\gamma_1},\tilde{\gamma}\in\cM_+(\R^{2d})^3 \times \cM_+(\R^{3d}):\\
&\qquad\gamma + {\pi_{0,1}}_{\#}\tilde{\gamma} \in \Pi(\nu_0, \nu_1), {\pi_{0,\mathrm{piv}}}_\# \tilde{\gamma}=\tilde{\gamma_0},{\pi_{\mathrm{piv},1}}_\# \tilde{\gamma} = \tilde{\gamma_1}\Big\}.
\end{split}\]
As before, a location constraint on a given set $K$ can be encoded in the choice of the reference measures:
$$R=e^{-c_1/\eps}\nu_0\otimes\nu_1,\quad R_0=e^{-c_0/\eps}\nu_0\otimes\mathds{1}_K,\quad R_1=e^{-c_1/\eps}\mathds{1}_K \otimes \nu_1.$$
For the density constraint, we have to add the condition ${\pi_1}_\# \tilde{\gamma_0} \leq \phi$. The dual of \eqref{eq:parkingentrreg} in the case of a density constraint is then given by 
\begin{eqnarray*}
\sup_{\substack{\varphi_0,\varphi_1,\\ \psi_0, \psi_1}} \left\{
-\int_{\R^{2d}} \exp\left(\varphi_0 + \varphi_1 \right)\d R
-\sum_{i=0}^1 \int_{\R^{2d}} \exp\left(\varphi_i + \psi_i \right)\d R_i\right.\\
\left. +\sum_{i=0}^1 \int_{\R^d} \varphi_i \d \nu_i+\int_{\R^d} \left( \psi_0 + \psi_1\right) \phi \d x\ :\ \psi_0 + \psi_1\le0\right\}.
\end{eqnarray*}
The Sinkhorn iterations (density constraint included) in the dual variables then become
\[\begin{split}
\exp\left(\varphi_i^{l+1}(x_i)\right) & =\left( \int \exp\left(-\frac{c_1}{\eps} + \varphi_{i + 1 \mod 2}^l(x) \right) \d x + \int \exp\left(-\frac{c_i}{\eps} + \psi_i^l(x) \right) \d x \right)^{-1},\\
\exp\left(\psi_i^{l+1}(x)\right) & =\min\left\{\mu^l,\phi\right\} \left( \int \exp\left(-\frac{c_i}{\eps} + \varphi_{i}^{l+1}(x_i) \right) \d \mu_i (x_i) \right)^{-1},
\end{split}\]
and $\mu^l$, the current approximate parking measure, is again given by an explicit geometric mean expression.

\bigskip

\subsection{Numerical results: comparison of the optimal interpolation and the optimal parking} We now present some numerical results based on the iterative schemes described in the previous paragraph. In all our examples (presented in Figures 4 to 7), we compare the solutions of the interpolation and parking problems with a constant density constraint on the unit square $K=[0,1]^2$. We always take as distribution of services $\mu_1=\nu_1=\delta_{(0.5, 0.5)}$, the Dirac at the center of the square and as distribution of residents, we take a symmetric sum of four Dirac masses:
\[\mu_0=\nu_0=\frac14\Big(\delta_{(0.5, 0.1)}+\delta_{(0.5,0.9)}+\delta_{(0.1,0.5)}+\delta_{(0.9,0.5)}\Big)\]
We consider power-like costs
\[c_0(x,y)=|x-y|^p,\qquad c_1(x,y)=2\,c_0(x,y)\]
for several values of $p$ corresponding to concave, linear or convex costs and various constant threshold values for the density constraints $\phi$. In this setting, we know (Corollary \ref{corodiscret} for $p>1$ and Proposition \ref{bangbangdistlike} for $p\leq 1$) that the optimal interpolation and the optimal parking are of bang-bang type. Even with the entropic regularization (which has the effect of blurring the true solution) this is clearly what we observe in these figures with a small regularization $\eps=5.10^{-4}$. Since the optimal parking may have total mass less than $1$, we have indicated its total mass on each figure, of course if the total mass of the parking is $1$ it coincides with the interpolation, a case which is more likely to occur when the threshold level is high. Finally, one can see the influence of the exponent $p$ on the shape of the support of the optimal measure and in particular recognize for $p=1$ (Figure 6) the drop-like shape which was explicitly computed and plotted in Figure 3 and balls for $p=2$ (Figure 7). 

\begin{figure}[h!]
	\centering
	\begin{subfigure}{\textwidth}
	\includegraphics[scale=0.5]{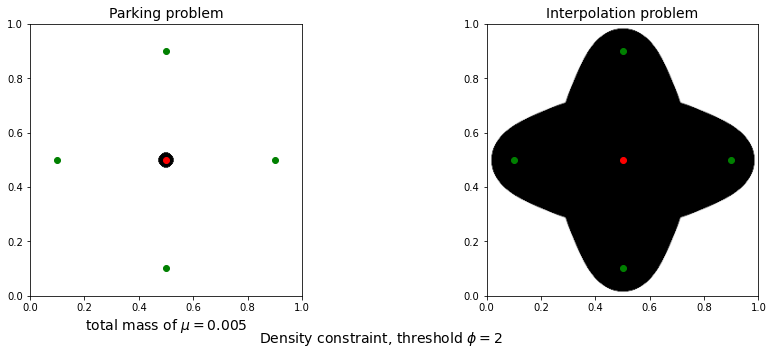}
	\end{subfigure}
	\begin{subfigure}{\textwidth}
	\includegraphics[scale=0.5]{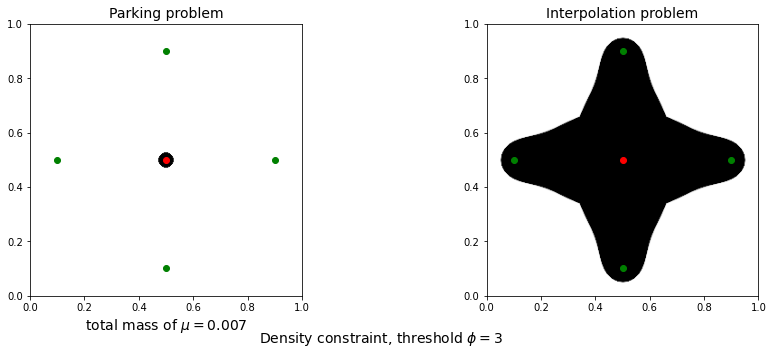}
	\end{subfigure}
	\begin{subfigure}{\textwidth}
	\includegraphics[scale=0.5]{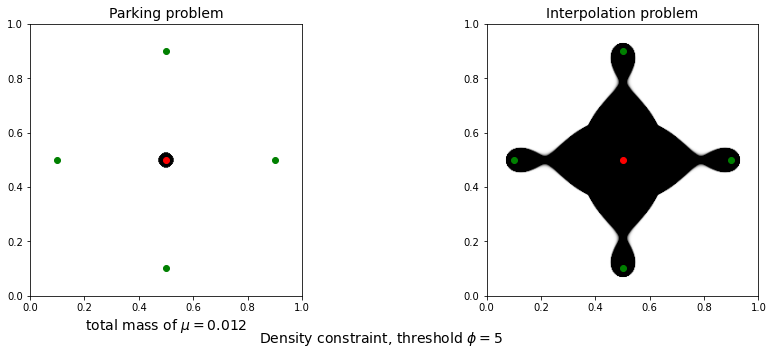}
	\end{subfigure}
	\begin{subfigure}{\textwidth}
	\includegraphics[scale=0.5]{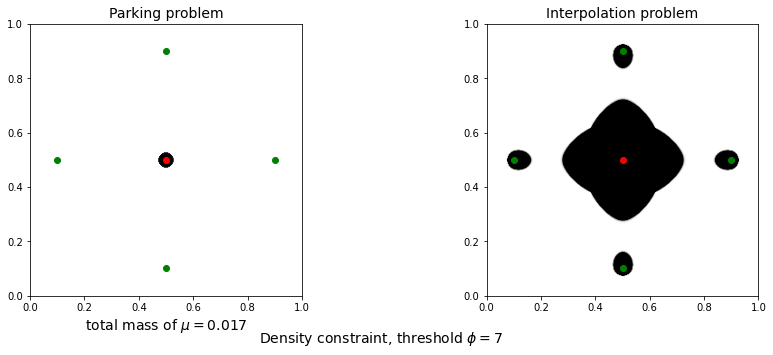}
	\end{subfigure}
	\caption{concave cost $p=0.25$}
	\end{figure}
	
\begin{figure}[h!]
	\centering
	\begin{subfigure}{\textwidth}
	\includegraphics[scale=0.5]{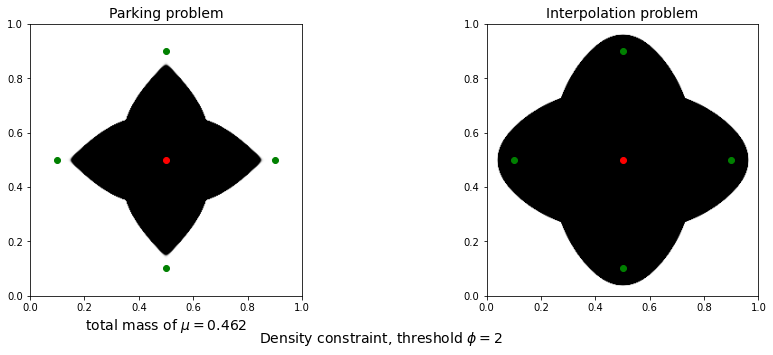}
	\end{subfigure}
	\begin{subfigure}{\textwidth}
	\includegraphics[scale=0.5]{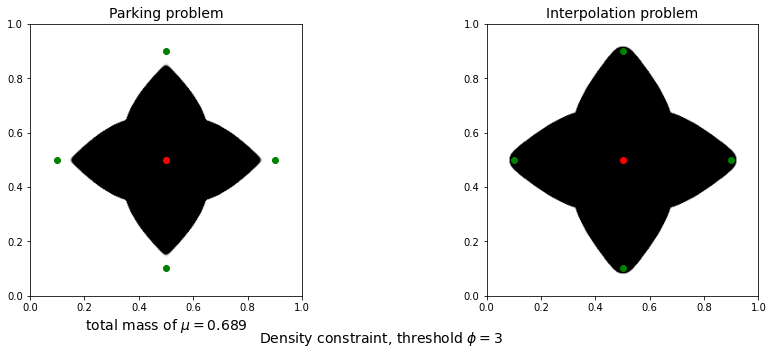}
	\end{subfigure}
	\begin{subfigure}{\textwidth}
	\includegraphics[scale=0.5]{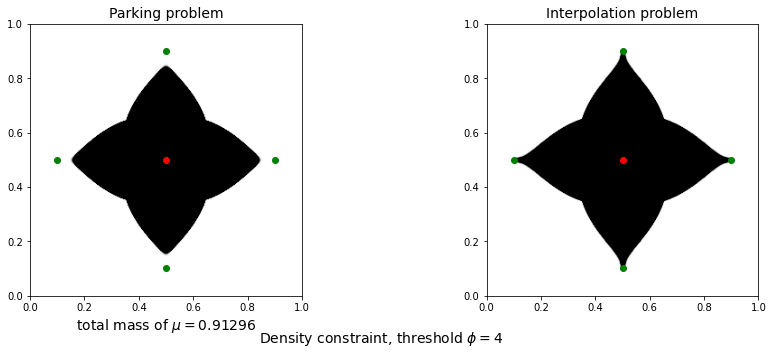}
	\end{subfigure}
	\begin{subfigure}{\textwidth}
	\includegraphics[scale=0.5]{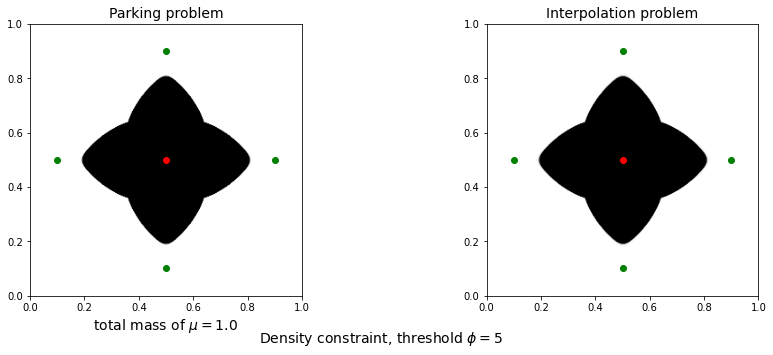}
	\end{subfigure}
	\caption{concave cost $p=0.75$}
\end{figure}

\begin{figure}[h!]
	\centering
	\begin{subfigure}{\textwidth}
	\includegraphics[scale=0.5]{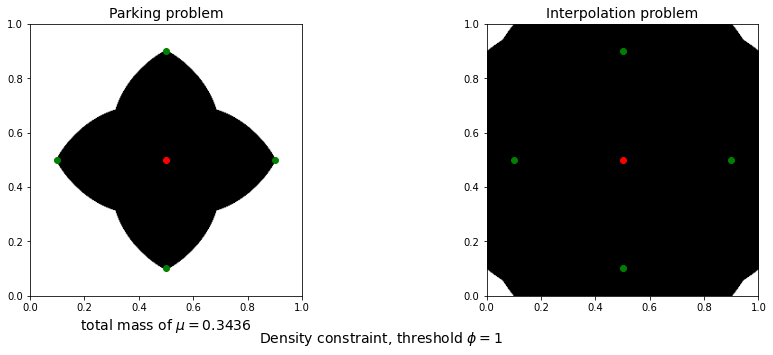}
	\end{subfigure}
	\begin{subfigure}{\textwidth}
	\includegraphics[scale=0.5]{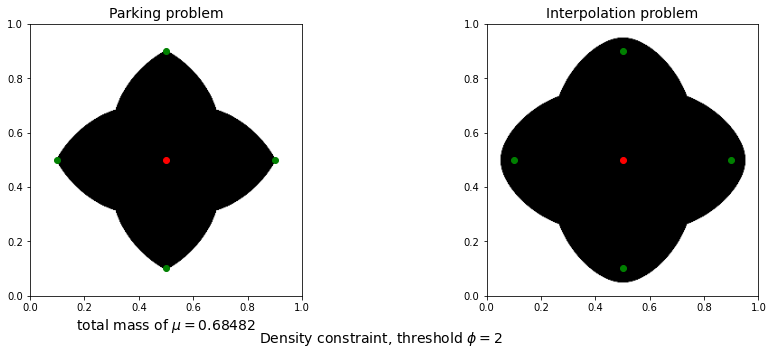}
	\end{subfigure}
	\begin{subfigure}{\textwidth}
	\includegraphics[scale=0.5]{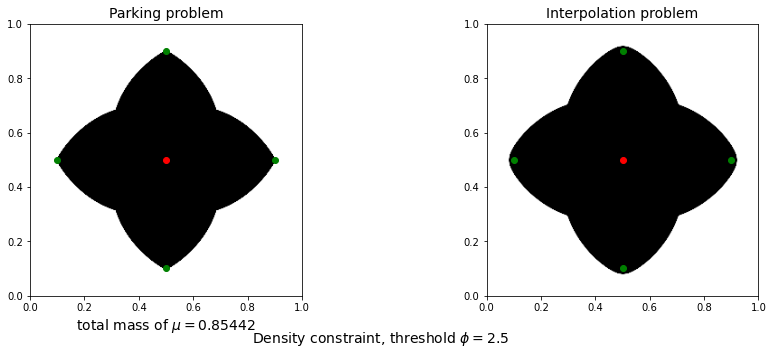}
	\end{subfigure}
	\begin{subfigure}{\textwidth}
	\includegraphics[scale=0.5]{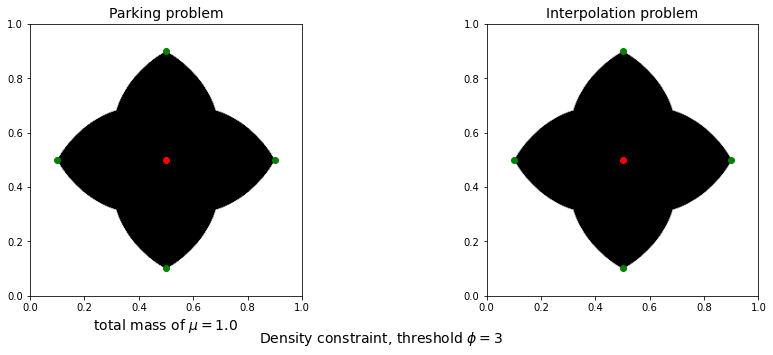}
	\end{subfigure}
	\caption{Linear cost $p=1$}
\end{figure}	

\begin{figure}[h!]
	\centering
	\begin{subfigure}{\textwidth}
	\includegraphics[scale=0.5]{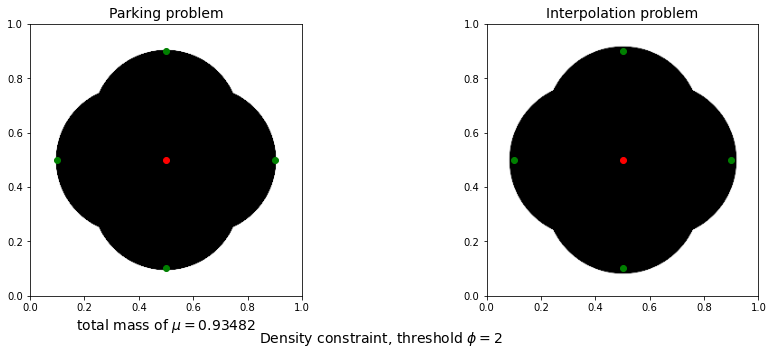}
	\end{subfigure}
	\begin{subfigure}{\textwidth}
	\includegraphics[scale=0.5]{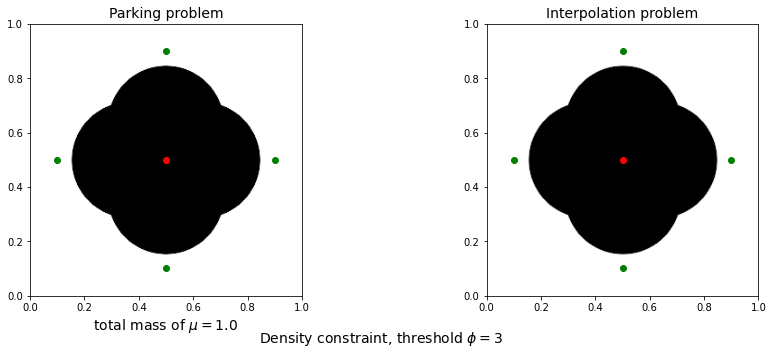}
	\end{subfigure}
	\begin{subfigure}{\textwidth}
	\includegraphics[scale=0.5]{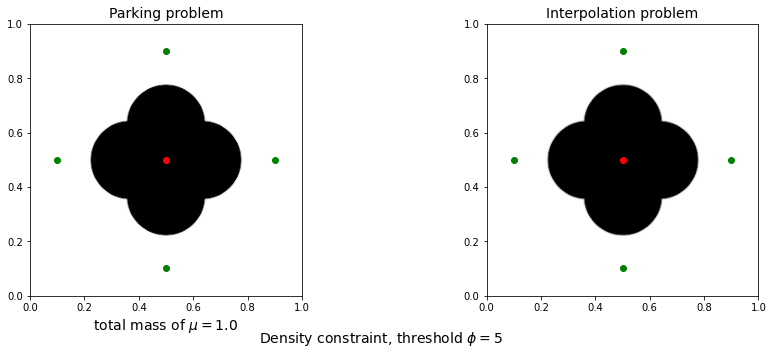}
	\end{subfigure}
	\begin{subfigure}{\textwidth}
	\includegraphics[scale=0.5]{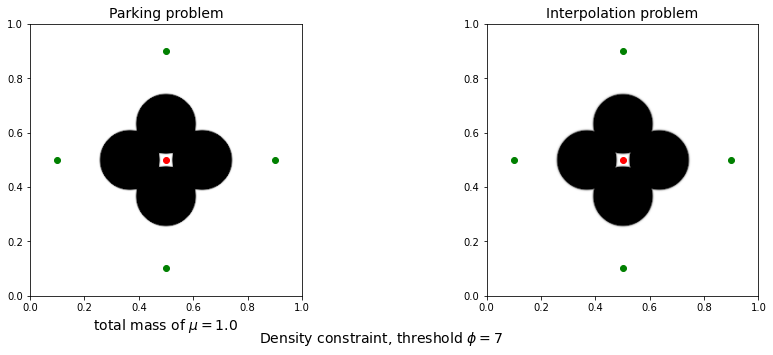}
	\end{subfigure}
	\caption{Convex cost $p=2$}
\end{figure}	

\bigskip

\noindent{\bf Acknowledgments.} The authors are grateful to an anonymous referee for insightful and kind suggestions that directly inspired Lemma \ref{lemmunontrivial}. The first author is member of the Gruppo Nazionale per l'Analisi Matematica, la Probabilit\`a e le loro Applicazioni (GNAMPA) of the Istituto Nazionale di Alta Matematica (INdAM); his work is part of the project 2017TEXA3H {\it``Gradient flows, Optimal Transport and Metric Measure Structures''} funded by the Italian Ministry of Research and University. This work was initiated during a visit of GC and KE at the Dipartimento di Matematica of the University of Pisa, the hospitality of this institution is gratefully acknowledged as well as the support from the Agence Nationale de la Recherche through the project MAGA (ANR-16-CE40-0014). G.C. acknowledges the support of the Lagrange Mathematics and Computing Research Center.
KE acknowledges that this project has received funding from the European Union's Horizon 2020 research and innovation programme under the Marie Sk\l odowska-Curie grant agreement No 754362.
\includegraphics[width=0.6cm]{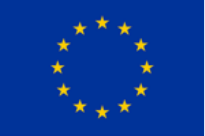}
\newpage
\bibliographystyle{plain}
\bibliography{bibli}

\bigskip
{\small\noindent
Giuseppe Buttazzo:
Dipartimento di Matematica, Universit\`a di Pisa\\
Largo B. Pontecorvo 5, 56127 Pisa - ITALY\\
{\tt giuseppe.buttazzo@unipi.it}\\
{\tt http://www.dm.unipi.it/pages/buttazzo/}

\bigskip\noindent
Guillaume Carlier:
CEREMADE, Universit\'e Paris Dauphine and Inria-Mokaplan, PSL\\
Place du Mar\'echal de Lattre de Tassigny, 75775 Paris Cedex 16 - FRANCE\\
{\tt carlier@ceremade.dauphine.fr}\\
{\tt https://www.ceremade.dauphine.fr/~carlier/}

\bigskip\noindent
Katharina Eichinger:
CEREMADE, Universit\'e Paris Dauphine and Inria-Mokaplan, PSL\\
Place du Mar\'echal de Lattre de Tassigny, 75775 Paris Cedex 16 - FRANCE\\
{\tt eichinger@ceremade.dauphine.fr}

\end{document}